\documentclass{amsart}
\usepackage[vcentermath,enableskew]{youngtab}
\usepackage{epsfig}
\usepackage{graphicx}
\usepackage{eepic}
\usepackage{amssymb,latexsym,ifthen}
\usepackage{exscale,relsize}
\newtheorem{Def}{Definition}
\usepackage[all]{xy}
\xyoption{matrix}
\xyoption{arrow}

\newtheorem{Thm}{Theorem}

\newtheorem{Rem}{Remark}
\newtheorem{Prop}{Proposition}

\newtheorem{lemma}{Lemma}

\newcommand{\s}{\scriptscriptstyle}
\newcommand{\D}{\displaystyle}

\newtheorem{Cor}{Corollary}
\newtheorem{Conj}{Conjecture}

\newcommand{\Path}{\operatorname{Path}}
\newcommand{\chess}{\operatorname{Chess}}
\newcommand{\super}{\operatorname{Tab}}

\newcommand{\wt}{\operatorname{wt}}
\newcommand{\Res}{\operatorname{Res}}

\newcommand{\res}{\operatorname{res}}
\newcommand{\SL}{\operatorname{SL}}

\newcommand{\GL}{\operatorname{GL}}

\newcommand{\pr}{\operatorname{pr}}
\newcommand{\gr}{\operatorname{gr}}
\newcommand{\Tab}{\operatorname{Tab}}

\newcommand{\row}{\operatorname{row}}
\newcommand{\col}{\operatorname{col}}
\newcommand{\set}{\operatorname{set}}
\newcommand{\ten}{10}
\newcommand{\eleven}{11}
\newcommand{\twelve}{12}
\title[Block-Toeplitz determinants, chess tableaux,
type $\widehat{A_1}$ GLS $\varphi$-map]
{Block-Toeplitz determinants, chess tableaux, and
the type $\widehat{A_1}$ Geiss-Leclerc-Schr\"oer $\varphi$-map}

\author[Jeanne Scott]{Jeanne Scott} 
\address{School of Mathematics\\
University of Leeds\\
Leeds LS2 9JT\\
United Kingdom}
\email{jscott@maths.leeds.ac.uk}

\thanks{Thanks to both Robert Marsh 
and Bill Crawley-Boevey
for discussion, their advice,
and help during the drafting of this 
paper. The author was supported by 
EPSRC grant EP/C01040X/1} 

\begin{document}

\maketitle

\begin{abstract}

\bigskip
\bigskip
\noindent
We evaluate the Geiss-Leclerc-Schr\"oer $\varphi$-map
for {\it shape modules}  
over the preprojective algebra $\Lambda$ of type $\widehat{A_1}$
in terms of matrix minors arising from the 
block-Toeplitz representation of 
the loop group $\SL_2(\mathcal{L})$. 
Conjecturally these minors are among the  
cluster variables for coordinate rings of 
unipotent cells within $\SL_2(\mathcal{L})$.
In so doing we compute the Euler
characteristic of any {\it generalized flag variety}
attached to a shape module by counting standard tableaux
of requisite shape and {\it parity}; alternatively by
counting {\it chess tableaux} of requisite shape and
content.

\end{abstract}

\bigskip
\section*{Introduction:}

\bigskip
\bigskip
\noindent
In \cite{Rigid} C. Geiss, B. Leclerc, and J. Schr\"oer
initiated the study of the (generalized) tilting theory for
preprojective algebras of Dynkin type $\Delta$ vis \'a  vie
the cluster algebra structure  
of the coordinate ring of the maximal unipotent
group $U_+$ attached to $\Delta$ under the Cartan-Killing classification;
see \cite{CA3}, \cite{CA1}, and \cite{CA2} regarding {\it cluster
algebras}.
Specifically they construct
an explicit map $\varphi$ from the module category of the
preprojective algebra to the coordinate ring of the corresponding
maximal unipotent group
which transforms exceptional objects into cluster variables
and maximal rigid modules into clusters. 
The $\varphi$-map can be interpreted (and this is the view
initially taken here) as type of
partition function which records the Euler characteristics
of {\it generalized flag varieties} attached to the module.

\bigskip
\noindent
Recently
both \cite{GLS2} and \cite{BIRS} have independently 
made the new step of examining the GLS $\varphi$-map in  
the affine setting. In particular \cite{BIRS} 
studied examples of unipotent cells of the loop group
$\SL_2(\mathcal{L})$ and proved that their coordinate
rings --- in accordance with the predictions made in
\cite{CA3} --- are cluster algebras of geometric type.
This was accomplished in part by evaluating the type
$\widehat{A_1}$ GLS $\varphi$-map for a fixed family of
nilpotent finite dimensional modules over the preprojective
algebra of type $\widehat{A_1}$. For each unipotent
cell in $\SL_2(\mathcal{L})$ the authors of \cite{BIRS}
conjecture 
an explicit list of nilpotent $\Lambda$-modules
whose images under the type $\widehat{A_1}$
GLS $\varphi$-map form the initial seed generating the
cluster algebra structure of the coordinate ring 
of the unipotent cell. Moreover  
conjecture (4.3) of \cite{BIRS} predicts 
determinantal expressions for these initial
cluster variables.

\bigskip
\noindent
This paper evaluates 
the type $\widehat{A_1}$ GLS $\varphi$-map
over a class of
nilpotent $\Lambda$-modules called {\it shape
modules} ---
indeed a class which properly contains those modules  
stipulated in conjectures (4.1)-(4.3) of \cite{BIRS}
--- and expresses the result determinantally
in order to settle (4.3) of \cite{BIRS}. The
proof entails computing the Euler characteristic
of any generalized flag variety attached to a shape
module; this is accomplished combinatorially
by counting standard tableaux of requisite shape
and parity. We now go into more detail:

\bigskip
\noindent
Recall that the preprojective algebra $\Lambda$ 
of type $\widehat{A_1}$ is defined as the  
quotient of the {\it path algebra}
associated to the quiver $Q$

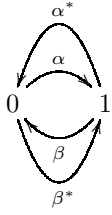
\[ \xymatrix{ 0 \ar@/^1pc/[r]^{\alpha} \ar@/_2.5pc/[r]_{\ \beta^*}  
& 1 \ar@/^1.1pc/[l]^{\beta}  \ar@/_2.5pc/[l]_{\ \alpha^*} } \]

\bigskip
\bigskip
\noindent
by the ideal $I$ generated by $\alpha^*\alpha - \beta \beta^*$ and 
$\beta^*\beta - \alpha \alpha^*$. Let $e_0$ and $e_1$ denote the
idempotents of $\Lambda$. 

\bigskip
\noindent
For a finite dimensional left $\Lambda$-module $M$
of dimension $\dim M = n$ together with a choice
of bit string ${\bf d} = (d_1, \dots, d_n)$ in
$\{0,1\}^n$ the {\it generalized flag variety}
$\mathcal{F}^{\s \Lambda}_{\s \bf d}\big(M\big)$ 
is the variety of all $\Lambda$-composition
series $M_n \supset \cdots
\supset M_0$ with $M_n = M$ and $M_0 = \{ 0 \}$
such that $M_t/M_{t-1} \simeq S_{d_t}$ whenever
$n \geq t \geq 1$. Here $S_0$ and $S_1$ 
are the simple left $\Lambda$-modules 
associated to the vertices labeled $0$ and $1$ 
in the quiver $Q$.

\begin{Def}[Type $\widehat{A_1}$ GLS $\varphi$-map]
Let ${\bf i} = (i_1, \dots, i_k)$ be 
an alternating bit string in $\{0,1\}^k$
and let $a_1, \dots, a_k$ denote
--- for the moment --- formal variables.
If $M$ is a finite dimensional left $\Lambda$-module then

\[ \varphi_{\s M} (a_1, \dots, a_k) \ 
:= \ \ {\D \sum_{ \textstyle  \ \ {\bf j} \in \Bbb{Z}_{\s \geq 0}^k }  
\chi \bigg( 
\mathcal{F}^{\s \Lambda}_{\bf  i^j } (M) \bigg) \ 
{a_1^{j_1} \cdots a_k^{j_k} \over {j_1! \cdots j_k!}} } \]

\bigskip
\noindent
where ${\bf i}^{\bf j}$ is the bit string in $\{0,1\}^n$ given by

\[ \big( i_{\sigma(1)}, \dots, i_{\sigma(n)} \big) \]

\bigskip
\noindent
with $n = j_1 + \dots + j_k$ and where 

\begin{equation}
\sigma(t) := \min \big\{ s \, \big| \, j_1 + \dots + j_s \geq t \big\} 
\end{equation}

\bigskip
\noindent
whenever $n \geq t \geq 1$. The symbol $\chi$ denotes Euler characteristic
(for cohomology with compact support; see \cite{toric}).

\end{Def}

\bigskip
\noindent
Section (1) of this paper begins with a quick 
survey of partition and
tableau combinatorics. The notion of
$i$-{\it parity} of a standard tableau is defined 
together with the auxiliary notion of a {\it chess tableau}.
Proposition (1) proves that 
the number of standard tableaux of shape $\lambda$ and 
$i$-parity ${\bf i}^{\bf j}$ equals $j_1! \cdots j_k!$
times the number of 
chess tableaux of shape $\lambda$ and 
content ${\bf j}$ where ${\bf i}$ is an alternating bit string
in $\{0,1\}^k$ and ${\bf j} \in \Bbb{Z}_{\s \geq 0}^k$.

\bigskip
\noindent
In section (2) we give a construction which
associates to a pair $\mu \subset \lambda$ of 
ordered partitions and
choice of parity $i =0, 1$ a nilpotent left
$\Lambda$-module called the {\it skew-shape module}.
The modules considered in chapter (4) of 
\cite{BIRS} are particular examples.
{\it Shape modules} are defined as skew-modules
where the smaller partition $\mu$ is empty.
Following this section (3) details a proof of our first theorem:

\begin{Thm}
Let $M$ be a shape module of shape $\lambda$, 
parity $i$, and dimension $n$. If ${\bf d}
= (d_1, \dots, d_n)$ is a bit string in 
$\{0,1\}^n$ then the Euler characteristic
$\chi \Big( \, \mathcal{F}^{\s \Lambda}_{\s \bf d}(M) \, \Big)$
equals the number of standard tableaux $T$
of shape $\lambda$ whose $i$-parity equals ${\bf d}$.
\end{Thm} 

\bigskip
\noindent
Conjecture (1) of section (3) refines Theorem (1)
and tallies the number of $\Bbb{F}_q$-rational points
of $\mathcal{F}^{\Lambda}_{\s \bf d}\big(M\big)$ 
when $\Lambda$ is viewed as an algebra over a finite
field $\Bbb{F}_q$ with $q$ elements.

\bigskip
\noindent
The {\it algebraic loop group} $\SL_2(\mathcal{L})$ is the group consisting of all
$\mathcal{L}$-valued $2 \times 2$ matrices $g = \big(g_{ij}\big)$ with determinant 
$1$ where $\mathcal{L}$ is the Laurent polynomial ring 
$\Bbb{C}\big[t,t^{-1} \big]$. An element $g \in \SL_2(\mathcal{L})$
is viewed as encoding a regular 
map $g:\Bbb{C}^* \longrightarrow SL_2(\Bbb{C})$ given by 
$z \mapsto \big(g_{ij}(z)\big)$; a closed contour or ``loop'' 
in $\SL_2(\Bbb{C})$ 
is obtained upon restricting 
the map to the circle group $S^1$ in $\Bbb{C}^*$ and taking its image
--- hence the name.

\bigskip
\noindent
The {\it maximal unipotent subgroup} 
$U_+$ is the subgroup of $\SL_2(\mathcal{L})$ containing all loops $g: \Bbb{C}^*
\longrightarrow \SL_2(\Bbb{C})$ which extend to $0$ and for which 
$g(0)$ is an upper triangular unipotent matrix. As $\mathcal{L}$-valued matrices

\[ U_+  \ = \ \Bigg\{ g \in \begin{pmatrix}
1 + t\Bbb{C}[t] & \Bbb{C}[t] \\ \\ t\Bbb{C}[t] & 1 + t\Bbb{C}[t] 
\end{pmatrix} \ \Bigg| \ \text{with $\det(g) = 1$} \ \Bigg\}. \]

\bigskip
\noindent
For $i=0,1$ and $a \in \Bbb{C}^*$ 
let $x_i: \Bbb{C}^* \longrightarrow
U_+$ denote the 1-parameter subgroups defined by

\[ x_0(a):= \begin{pmatrix} 1 & 0 \\ at & 1 \end{pmatrix} \qquad
x_1(a):= \begin{pmatrix} 1 & a \\ 0 & 1 \end{pmatrix} \]

\bigskip
\noindent
The set of elements $\mathcal{O}_+$ 
which factorize as $x_{i_1}(a_1) \cdots x_{i_k}(a_k)$ for
some alternating bit string ${\bf i} = (i_1, \dots, i_k)$ 
in $\{0,1\}^k$, for some choice of parameters $a_1, \dots, a_k$
in $\Bbb{C}^*$, and for some $k$
is a Zariski open subset within $U_+$. Consequently
a regular function over $U_+$ is uniquely determined
by its values over $\mathcal{O}_+$.

\bigskip
\noindent
In section (4) we define regular functions 
$\Delta^{(i)}_{\mu, \lambda}: \SL_2(\mathcal{L})
\longrightarrow \Bbb{C}$ indexed by a choice
of parity $i = 0,1$ together with a pair of partitions
$\mu$ and $\lambda$. These functions are expressed as
minors of the infinite block-Toeplitz matrix $T_g$
associated to the argument $g \in \SL_2(\mathcal{L})$.
In particular the minors consider in \cite{BIRS}
are of the form $\Delta^{(i)}_{\mu ,\lambda}$.
These minors are shown to satisfy an $i$-Pieri rule 
--- reminiscent of the generalized Pieri identities
considered by \cite{Lam} --- and as
such behave like Schur polynomials which carry
a parity. This view is reinforced by  
remark (7) in section (5) which expresses
$\Delta^{(i)}_{{\s \emptyset}, \lambda}$ 
as a generating function for chess tableaux
of shape $\lambda$ and parity $i$.
Section (4) ends with  
conjecture (2) which claims  
that the type $\widehat{A_1}$-{\it generalized
minors} of Fomin-Zelevinsky (see \cite{Bruhat})
are among these block-Toeplitz matrix minors.

\bigskip
\noindent
In section (5) 
the restriction of the minors $\Delta^{(i)}_{\mu, \lambda}$
to the maximal unipotent subgroup $U_+$ of the
loop group are studied combinatorially by means of
pairwise non-crossing families of paths in a weighted
planar network $\Gamma_{\bf i}({\bf a})$. 
Proposition (3) sets up a weight and
content preserving bijection between families of
non-crossing paths and chess tableaux which is
then used to prove the main result of this paper:

\begin{Thm}
Let $M$ be a shape module of shape $\lambda$, parity $i$, 
and dimension $n$. Let 
${\bf i} = (i_1, \dots, i_k)$ be an alternating
bit string in $\{0,1\}^k$ and let $a_1, \dots, a_k$
be parameters in $\Bbb{C}^*$ then

\[ 
{\mathlarger{\mathlarger{\mathlarger\Delta}}}_{ {\s \emptyset} , \lambda}^{(i)} 
\Big( x_{i_1}(a_1) \cdots x_{i_k}(a_k) \Big)
\ = \ \varphi_{\s M}\big(a_1, \dots, a_k\big) \]

\end{Thm}

\bigskip
\bigskip
\section{Partitions and Tableaux:}

\bigskip
\bigskip
\noindent
By a partition $\lambda$ we will mean a non-increasing 
infinite sequence
of non-negative integers $\big( \lambda_0 \geq \lambda_1 \geq \lambda_2 
\geq \dots
\big)$ with is eventually zero;
as short hand we only write the non-zero terms.
The size $|\lambda|$ of $\lambda$
is defined as the sum $\lambda_0 + \lambda_1 + \dots$ and
we say the $\lambda$ is a partition of $m$ if 
$|\lambda| = m$.
In the context of partitions the symbol $\emptyset$ will 
denote the {\it empty partition} defined by $\lambda_n = 0$
whenever $n \geq 0$. 
For a partition $\lambda$ let $N_\lambda := \max \big\{ 
n \, \big| \, \lambda_n > 0 \big\}$ with the understanding
that $N_\emptyset = 0$. Given a partition $\lambda$, a
choice of parity $i = 0,1$, and a non-negative integer 
$N$ define the $\set^{(i)}_N (\lambda)$ by

\[ \set^{(i)}_N (\lambda) \ := \ \big\{ \lambda_n + i -n \,
\big| \, N \geq n \geq 0 \big\} \]

\bigskip
\noindent
In this paper a partition $\lambda$ of $n$ will be synonymous
with its Young diagram or {\it shape}: This is a
left-justified arrangement of $n$ boxes into rows 
whereby the top row contains $\lambda_0$ boxes, the
next $\lambda_1$, and so on. For instance shape of the
partition $\lambda = (4,3,2,2,1)$ is

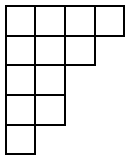
\[ \Yvcentermath1 \yng(4,3,2,2,1) \]

\bigskip
\noindent
Each box in a shape will be coordinatized by its
row and column positions $s,t$ measured from the
top and left respectively. We employ the
convention that the top row and left-most
column are counted as row and column zero.

\begin{Def} Let $\lambda$ be a partition
and let $i = 0, 1$
be a choice of parity. The {\bf $i$-parity} of a box with
row and column coordinates $s,t$ in the 
the shape of $\lambda$ is by
definition the parity of $s+t+i$.
\end{Def}

\bigskip
\noindent
By definition two partitions $\mu = (\mu_0 \geq \mu_1
\geq \dots )$ and $\lambda = (\lambda_0 \geq \lambda_1
\geq \dots)$ are ordered $\mu \subseteq \lambda$ if
$\mu_t \leq \lambda_t$ for all integers
$t \geq 0$. The symbol $\subseteq$ is justified by
the fact that $\mu \subseteq \lambda$ if and only if
the shape of $\mu$ is contained in the shape of $\lambda$.
In this case let $\lambda/ \mu $ denote the {\it 
skew-shape} obtained by removing $\mu$ from $\lambda$ as shapes. 
For example: If $\mu = (2,1)$ and $\lambda = (4,3,2,2,1)$ then $\mu$,
$\lambda$, and $\lambda / \mu$ are depicted respectively by

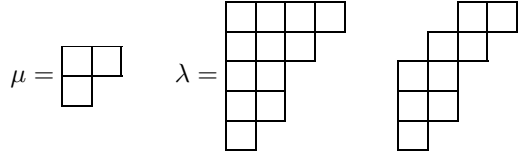
\[ \Yvcentermath1 \mu = \yng(2,1) \qquad \lambda = \yng(4,3,2,2,1) 
\qquad 
\young(::\hfil \hfil,:\hfil \hfil,\hfil \hfil,\hfil \hfil,\hfil) \]

\bigskip
\noindent
Let $\lambda$ be a partition and ${\bf q}=(q_1, \dots, q_m)$
be a non-negative integer $m$-tuple in $\Bbb{Z}^m_{\s \geq 0}$
such that $q_1 + \dots + q_m = |\lambda |$.
For the purposes of this paper, 
a {\it semi-standard tableau} $T$ of shape $\lambda$ 
and content ${\bf q}$ is an assignment whereby each
box in shape of $\lambda$ is labeled by an  
integer $t$ in $[1 \dots m]$ such that 

\bigskip
\indent \indent $\bullet$ each index $t \in [1 \dots m]$
is used exactly $q_t$ times

\indent \indent $\bullet$ the indices $t$ strictly increase when read
from left to right in each row 

\indent \indent $\bullet$ the indices $t$ weakly increase 
when read from top to bottom in each column

\bigskip
\noindent
In the event $n \geq |\lambda|$ and ${\bf q} =
(q_1, \dots, q_n)$ is
a bit string in $\{0,1\}^n$ with
$q_1 + \dots + q_n = |\lambda|$ we say that 
$T$ is a {\it standard tableau of content ${\bf q}$}.
Note that if in this case the indices 
$t$ of $T$ are read from top to bottom in each column they 
will necessarily strictly increase. 
By convention a standard tableau $T$ of
shape $\lambda$ 
without any reference to content will 
be understood to be a standard tableau
of shape $\lambda$ and content ${\bf q}
\in \{0,1\}^{|\lambda|}$ where $q_t = 1$ whenever
$|\lambda| \geq t \geq 1$.
Below are examples of 
a standard 
and a semi-standard tableaux 
of shape $\lambda = (4,3,2,2,1)$
with the later having content  
${\bf q} = (2,1,1,3,2,2,1)$:

\[  \Yboxdim{15pt}  
\young(1257,369,4\ten,8\twelve,\eleven)
\qquad \qquad 
\young(1245,134,46,56,7) \]

\bigskip
\noindent
For a partition $\lambda$ and a bit string
${\bf q}$ in $\{0,1\}^n$ with $n \geq |\lambda|$
and $q_1 + \dots + q_n = |\lambda|$ 
the set of all standard tableaux
of shape $\lambda$ and content ${\bf q}$ will be
denoted $\Tab_{\s \bf q}(\lambda)$; the
set of all standard tableaux of shape
$\lambda$ will be denoted $\Tab(\lambda)$

\bigskip
\noindent
Given an arbitrary non-negative integer $m$-tuple 
${\bf q} = (q_1, \dots, q_m)$ in $\Bbb{Z}^m_{\s \geq 0}$
let $\bf [q]$ denote the {\it indicator set}, i.e.

\[ {\bf [q]} \ := \ \big\{ t \, \big| \, q_t \ne 0 \big\}  \]

\bigskip
\noindent
For $n \geq m$ let $\{0,1\}^n_m$ denote the set of bit
strings ${\bf q}$ such that ${\bf [q]}$ has cardinality 
$m$. We may express the indicator set of
a bit string ${\bf q}$ in $\{0,1\}^n_m$ 
as ${\bf [q]} = \big\{ r_1 < \cdots < r_m \big\}$.
For a partition $\lambda$ of $m$ 
and a tableau $T$ in $\Tab_{\s \bf q}(\lambda)$ let
$\overline{T}$ be the tableau in 
$\Tab(\lambda)$ obtained by replacing each $r_t$ in $T$ 
by $t$ whenever $m \geq t \geq 1$. Clearly the mapping
$T \longmapsto \overline{T}$ defines a bijection
between $\Tab_{\s \bf q}(\lambda)$ and
$\Tab(\lambda)$.

\bigskip
\noindent

\bigskip
\noindent
Assume $\lambda$ is a partition of $m$ and 
fix a bit string ${\bf q} \in \{0,1\}^n_m$
and express it indicator set as ${\bf [q]}
= \big\{ r_1 < \cdots < r_m \big\}$.
A {\it ${\bf q}$-step flag of partitions} in $\lambda$ is a
weakly increasing  
sequence of partitions 
$\lambda^{(0)} \subseteq \cdots \subseteq \lambda^{(n)}$ 
with $\lambda^{(0)} = \emptyset$
and $\lambda^{(n)} = \lambda$ 
such that 
the shapes of $\lambda^{(t-1)}$ and $\lambda^{(t)}$ 
differ by exactly one corner box whenever $q_t = 1$
and are equal whenever $q_t = 0$.
Associate to a ${\bf q}$-step flag 
of partitions $\big( \lambda^{(t)} \big)_{\s n \geq t \geq 0}$
a standard tableau $T$ with content ${\bf q}$ obtained by
labeling the box of $\lambda$ which is deleted from $\lambda^{(t)}$
in order to get $\lambda^{(t-1)}$ by $r_t$.
For example when $\lambda = (3,2,1)$ 
and ${\bf q}= (1,1,0,1,1,1,0,1)$
the ${\bf q}$-step flag of partitions (depicted here in terms of 
shapes)

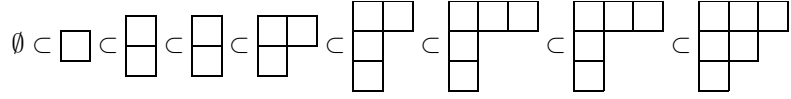
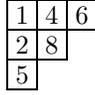
\[ \emptyset \subset \yng(1) \subset \yng(1,1) \subset \yng(1,1)
\subset \yng(2,1) \subset \yng(2,1,1) \subset \yng(3,1,1) \subset
\yng(3,1,1) \subset \yng(3,2,1) \]

\bigskip
\noindent
corresponds to 

\[ \young(146,28,5) \]

\bigskip
\noindent
The correspondence between ${\bf q}$-step flags of partitions in $\lambda$
and standard tableaux of shape $\lambda$ 
and content ${\bf q}$ is clearly bijective.

\begin{Def}
Let $\lambda$ be a partition of $n$ and let 
let $i=0, 1$ be a choice of parity.
The {\bf $i$-parity} of a standard tableau $T$ 
of shape $\lambda$ is by definition the bit string
${\bf d_{\s T}} = (d_1, \dots, d_n)$ in $\{0,1\}^n$
where $d_t$ equals 
the $i$-parity of the box labeled by $t$ in $T$.
Let $\super^{(i)} \big(\lambda; {\bf d}\big)$
denote the set of all such tableaux.  
\end{Def}

\begin{Def}
Let $\lambda$ be a partition, 
let $i=0,1$ be a choice
of parity, and let ${\bf j} = (j_1, \dots, j_k)$
be a non-negative integer $k$-tuple in $\Bbb{Z}^k_{\s \geq 0}$
such that $j_1 + \dots + j_k = |\lambda|$.
A {\bf chess tableau} $T$
of shape $\lambda$, parity $i$, 
and content ${\bf j}$ is 
a semi-standard tableau of shape $\lambda$
and content ${\bf j}$ 
with the added constraint 
that the $i$-parity of each box in $T$
equals the parity of the index $t$ 
labeling that box.
Let $\chess^{(i)}_{\s \bf j}(\lambda)$ 
denote the set of
all chess tableaux of shape $\lambda$, parity $i$, and
content ${\bf j}$. Let
$\chess^{(i)}(\lambda)$ denote the disjoint union

\[ \chess^{(i)}(\lambda) \ := \
\bigsqcup_{\D \ \, {\bf j} \in \Bbb{Z}^k_{\s \geq 0}} 
\chess^{(i)}_{\s \bf j}(\lambda) \]

\end{Def}

\begin{Rem}
The parity condition for chess tableaux forces 
the row and column entries to increase strictly.
\end{Rem}

\bigskip
\noindent
Chess tableaux seem to have been
considered first by Chow, Erisson, and Fan in
\cite{chess}. Among other things, they independently 
interpret chess tableaux
in terms of {\it rat races} which are essentially
the non-crossing path configurations which we
consider in section 5.

\begin{Prop}
Let $\lambda$ be a partition 
of $n$, let $i =0,1$ be a choice
of parity, let ${\bf d}$ be a bit string in $\{0,1\}^n$,
let ${\bf i}=(i_1, \dots, i_n)$ be an alternating
bit string in $\{0,1\}^n$, and let ${\bf j} = (j_1, \dots, j_k)$ be a non-negative
integer $k$-tuple in $\Bbb{Z}^k_{\s \geq 0}$
such that ${\bf i}^{\bf j} = {\bf d}$ then

\[ \bigg| \super^{(i)} \big(\lambda; {\bf d} \big) \bigg| 
\ = \ j_1! \cdots j_k! \
\bigg| \chess^{(i^*)}_{\s \bf j} \big( \lambda \big)               
\, \bigg| \]

\bigskip
\noindent
where $i^*$ is the parity of $i + i_1 + 1$.

\end{Prop}

\begin{proof}
Recall the definition of $\sigma(t)$ given in
equation (1) of the introduction.
Note that $\sigma(t) \in {\bf [j]}$ whenever $n \geq t \geq 1$.
Clearly $\sigma(t) \geq \sigma(t')$ whenever $t > t'$.
If $T$ be a tableau in 
$\super^{(i)} \big(\lambda; {\bf d} \big)$
then the filling $\Sigma(T)$ of $\lambda$ obtained by 
replacing each index $t$ in $T$ by 
$\sigma(t)$ is weakly increasing in both rows and columns. 
The parity of $\sigma(t)$ equals the parity of $d_t + i_1 + 1$
whenever $n \geq t \geq 1$. Consequently
the $i$-parity of the box labeled $t$ in $T$
equals $d_t$ if and only if the $i^*$-parity of the
box labeled $t$ in $T$ equals the parity of $\sigma(t)$.
By assumption the $i$-parity of the box 
in $T$ labeled $t$ equals $d_t$ whenever 
$n \geq t \geq 1$ therefore the $i^*$-parity
of the box labeled $\sigma(t)$ in $\Sigma(T)$ equals the parity
of $\sigma(t)$; hence $\Sigma(T) \in 
\chess^{(i^*)}\big(\lambda\big)$. Moreover for 
each $s \in [1 \dots k]$ 
the cardinality of $\sigma^{\s -1}(s)$ is exactly $j_s$ 
and so $\Sigma(T) \in 
\chess^{(i^*)}_{\s \bf j} \big(\lambda\big)$.

\bigskip
\noindent
If $S$ is a chess tableau in 
$\chess^{(i^*)}_{\s \bf j} \big(\lambda\big)$
let $T$ be a filling of $\lambda$ obtained
by replacing each label $s$ occurring in $S$ by
some choice of element $t$ in $\sigma^{\s -1}(s)$
with the rule that pre-images in 
$\sigma^{\s -1}(s)$ are not reused once they
are selected.
As indicated in the remark above, any chess tableau
is in fact row and column strict and since
$\sigma(t) \geq \sigma(t')$ whenever $t > t'$
it follows that $T$ will be a standard tableau.
Moreover the $i$-parity of the box labeled 
$t$ will be $d_t$. Consequently
the mapping $\Sigma: \super^{(i)} \big(\lambda; {\bf d} \big)
\longrightarrow 
\chess^{(i^*)}_{\s \bf j}\big(\lambda\big)$
is surjective. 

\bigskip
\noindent
Since the construction described in the
surjectivity argument prescribes that each
index in $s \in {\bf [j]}$ is chosen once
and since there are $j_s$-choices to be made
it follows that are exactly $j_s!$ possibilities for each
$s$. The proposition now follows.

\end{proof}

\bigskip
\noindent
For a standard tableau $T$ of shape
$\lambda$ and an index $t$ let
$\row_t$ and $\col_t$ respectively 
denote the row and column positions of the
box labeled by $t$ in $T$

\begin{Def}
Let $\lambda$ be a partition of $n$ and 
let ${\bf d} = (d_1, \dots, d_n)$ be a 
bit string in $\{0,1\}^n$.
A pair of indices $\{s, t\}$ 
is called a {\bf ${\bf d}$-transposition pair} for a standard 
tableau $T$ of shape $\lambda$
if $d_s = d_t$ and
the tableau $T'$ obtained
by exchanging the positions of $s$ and $t$ 
remains standard. A ${\bf d}$-transposition pair $\{s,t\}$ is said 
to be {\bf grounded} if $s< t$ and 
$\row_s < \row_t$
and $\col_t< \col_s$. Define the {\bf ground state}
$\gr\big(T\big)$ of a standard tableau $T$ as
the total number of grounded ${\bf d}$-transposition pairs of $T$.

\end{Def}

\bigskip
\section{Shape Modules over $\Lambda$:}

\bigskip
\bigskip
\noindent
We may always regard a finite dimensional left $\Lambda$-module $M$
as a module over the polynomial
ring $\Bbb{C}\big[ \delta \big]$ whereby $\delta := 
\alpha + \beta $. If in addition $M$ is
nilpotent as a $\Lambda$-module then its isomorphism type when viewed as a 
$\Bbb{C}\big[ \delta \big]$-module is determined by the partition 
$\lambda$ of $\dim M$ which encodes the Jordan type of 
$\delta = \alpha + \beta$ considered as a nilpotent endomorphism
on $M$; we call $\lambda$ the {\it $\Bbb{C}\big[\delta\big]$-partition
type} of $M$.
In this case we 
can visualize the $\Bbb{C}\big[ \delta \big]$-module structure
on $M$ using the shape of $\lambda$: Each
box in the shape corresponds to a basis vector in $M$
and each row corresponds to an 
indecomposable $\Bbb{C}\big[\delta\big]$-summand of $M$
--- with the convention that the action of $\delta$ 
is depicted as going from right to left in each row.
For example the shape and corresponding $\Bbb{C}\big[\delta\big]$-module
associated to the partition $\lambda = (4,2)$ are

\[  \lambda = \yng(4,2)
\quad \quad \xymatrix@-4mm{ 
v_{\s 0,0} 
& \ar[l]_{\s \delta} v_{\s 0,1} 
& \ar[l]_{\s \delta} v_{\s 0,2}
& \ar[l]_{\s \delta} v_{\s 0,3} \\ 
v_{\s 1,0} 
& \ar[l]_{\s \delta} v_{\s 1,1}
& 
& } \]

\bigskip
\noindent
where $v_{s,t}$ is the basis vector corresponding to the 
box in located in row $s$ and column $t$ of the shape 
of $\lambda$ --- with the proviso that the top row 
and left most column
of $\lambda$ both have coordinates equal to $0$. 

\bigskip
\noindent 
It is well 
known (e.g. see \cite{Macdonald} chapter II section 2)
that if $M$ is a nilpotent 
$\Bbb{C}\big[ \delta \big]$-module and $N$
is a proper $\Bbb{C}\big[\delta\big]$-submodule
then $\mu \subset \lambda$ where $\mu$ and 
$\lambda$ are the $\Bbb{C}\big[\delta\big]$-partition types of 
$N$ and $M$ respectively.

\bigskip
\noindent
For a finite dimensional left $\Lambda$-module $M$
of dimension $\dim M = n$ let 
$\mathcal{F}^{\s \delta}\big(M\big)$ denote
the variety of all $\Bbb{C}\big[ \delta \big]$-composition
series, i.e. complete flags $M_n \supset \cdots
\supset M_0$ with $M_n = M$ and $M_0 = \{ 0 \}$
such that $M_{t-1}$ is a maximal 
$\Bbb{C}\big[\delta\big]$-submodule of 
$M_t$ whenever $n \geq t \geq 1$. 

\bigskip
\noindent
If $M$ is a nilpotent $\Bbb{C}\big[\delta\big]$-module of
dimension $\dim M  =n$ and $\Bbb{C}\big[\delta\big]$-partition
type $\lambda$ then any
$\Bbb{C}\big[\delta\big]$-composition series 
$M_n \supset \cdots \supset M_0$ in $M$ gives rise by 
the remarks made earlier to a flag of
partitions $\lambda^{(n)} \supset \cdots \supset \lambda^{(0)}$
in $\lambda$ where $\lambda^{(t)}$ is the $\Bbb{C}\big[\delta\big]$-partition 
type of $M_t$. The standard tableau $T$ 
of shape $\lambda$ which records this flag of partitions
will be called the 
$\Bbb{C}\big[\delta\big]$-{\it tableau type}
of the $\Bbb{C}\big[\delta\big]$-composition series.

\begin{Def}
Let $M$ be a nilpotent left $\Lambda$-module of dimension
$\dim M = n$ and 
$\Bbb{C}\big[\delta\big]$-partition type $\lambda$. For a 
standard tableau $T$ of shape $\lambda$ let $\Omega^{\s \delta}_{\s T}
\big(M\big)$ denote the collection of all 
$\Bbb{C}\big[\delta\big]$-composition series of 
$\Bbb{C}\big[\delta\big]$-tableau type $T$ in
$\mathcal{F}^{\s \delta}\big(M\big)$. Set 
$\Omega^{\s \Lambda}_{\s T, {\bf d}}\big(M\big) := 
\Omega^{\s \delta}_{\s T}\big(M\big) \cap
\mathcal{F}^{\s \Lambda}_{\s \bf d}\big(M\big)$
for ${\bf d} \in \{0,1\}^n$.

\end{Def}

\bigskip
\noindent
Macdonald relates in \cite{Macdonald} that 
Spaltenstein proved in \cite{Spaltenstein} that 
$\Omega^{\s \delta}_{\s T}\big(M\big)$ is a smooth
irreducible locally closed subvariety of $\mathcal{F}^{\s \delta}
\big(M\big)$; moreover $\Omega^{\s \delta}_{\s T}\big(M\big)$
is a disjoint union of subvarieties each of which is
isomorphic to an affine space.

\begin{lemma} 
\[ \chi \Big( \mathcal{F}^{\s \Lambda}_{\s \bf d}\big(M\big) \Big)
\ = \ \sum_{T \in \Tab(\lambda)} \chi \Big( \,
\Omega^{\s \Lambda}_{\s T, \bf d}\big(M\big) \,
\Big) \]

\end{lemma}

\begin{proof}
The subset $\Omega^{\s \Lambda}_{\s T,{\bf d}}\big(M\big)$ is
a locally closed subvariety of $\mathcal{F}^{\s \Lambda}_{\s \bf d}
\big(M\big)$ owing to the fact that 
$\Omega^{\s \delta}_{\s T}\big(M\big)$ is locally closed in
$\mathcal{F}^{\s \delta}\big(M\big)$. 
The lemma must hold since Euler characteristic
$\chi$ (for cohomology with compact support;
see \cite{toric}) is additive over disjoint unions of locally closed 
subvarieties and because

\[ \mathcal{F}^{\s \Lambda}_{\s \bf d}\big(M\big)
\ = \ \bigsqcup_{T \in \Tab(\lambda)} 
\Omega^{\s \Lambda}_{\s T, \bf d}\big(M\big) \]

\end{proof}

\begin{Def}
Let $i= 0,1$ be a choice of parity and let $\lambda = (\lambda_0 \geq \lambda_1 \geq \dots)$
and $\mu = \big(\mu_0 \geq \mu_1 \geq \dots)$ be two partitions
such that $\mu \subset \lambda$.
The {\bf skew-shape module} $M$ of {\bf skew-shape} $\lambda / \mu$ and {\bf parity $i$}
is the nilpotent left $\Lambda$-module of dimension $|\lambda| - 
|\mu|$ with basis 

\[ \Big\{ v_{s,t} \, \Big| \, \mu_s \leq t < \lambda_s \, \Big\} \]

\bigskip
\noindent
for which the $\alpha$, $\beta$, $\alpha^*$, and $\beta^*$ actions are
given by

\[ \begin{array}{ll}
\alpha v_{s,t} = \left\{ \begin{array}{ll} 
v_{s-1,t} &\text{$s+t+i$ even} \\
0 &\text{otherwise} \end{array} \right.
&\alpha^*v_{s,t} = \left\{ \begin{array}{ll}
v_{s,t-1} &\text{$s+t+i$ odd} \\
0 &\text{otherwise} \end{array} \right. \\ \\
\beta v_{s,t} = \left\{ \begin{array}{ll}
v_{s-1,t} &\text{$s+t+i$ odd} \\
0 &\text{otherwise} \end{array} \right.
&\beta^*v_{s,t} = \left\{ \begin{array}{ll}
v_{s,t-1} &\text{$s+t+i$ even} \\
0 &\text{otherwise} \end{array} \right.
\end{array} \]

\bigskip
\noindent
and zero otherwise.
If $\mu = \emptyset$ we call $M$ instead a {\bf shape module}
of shape $\lambda$ and parity $i$. In addition the
shape module associated to $\lambda = \emptyset$
is by definition the zero module.

\end{Def}

\bigskip
\noindent
This module can be depicted by a grid of north and east
pointing arrows, each representing
the action of either $\alpha$, $\beta$, $\alpha^*$, and $\beta^*$,
which is subordinate to the skew-shape $\lambda / \mu$. For instance the case of $i=0$, $\mu = (2,1)$,
and $\lambda = (4,3,2,2,1)$ is shown here together with
the corresponding skew-shape:

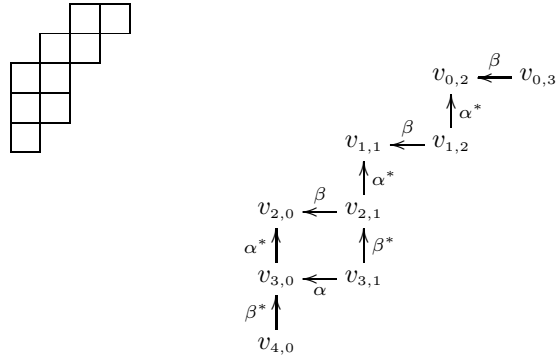
\[  \Yvcentermath1  
\young(::\hfil \hfil,:\hfil \hfil,\hfil \hfil,\hfil \hfil,\hfil)
\qquad \qquad \xymatrix@-4mm{ 
& 
& v_{\s 0,2} 
& v_{\s 0,3} \ar[l]_{\beta} \\ 
& v_{\s 1,1} 
& v_{\s 1,2} \ar[l]_{\beta} \ar[u]_{\alpha^*}
& \\
v_{\s 2,0}
& v_{\s 2,1} \ar[u]_{\alpha^*} \ar[l]_{\beta} 
&
& \\
v_{\s 3,0} \ar[u]^{\alpha^*}  
& v_{\s 3,1} \ar[l]^{\alpha} \ar[u]_{\beta^*}
& 
& \\
v_{\s 4,0} \ar[u]^{\beta^*}
& 
& 
& } \]

\begin{Rem}
The preprojective relations $\alpha^*\alpha - \beta \beta^* = 0$ 
and $\beta^*\beta - \alpha \alpha^* = 0$ are clearly
valid for any skew-shape module and any parity.
Clearly the $\Bbb{C}\big[\delta\big]$-partition type of the shape
module $M$ of shape $\lambda$ of any parity $i$ is $\lambda$.
\end{Rem}

\section{Evaluation of Euler Characteristics:}

\bigskip
\noindent
The techniques used in this section are adopted
from those described by Geiss-Leclerc-Schr\"oer in 
sections 2 and 3 of their publication \cite{GLS}. 
In addition we have tried to reconcile our notation
with theirs as much as possible.

\bigskip
\noindent
Let $M$ be a shape module of shape $\lambda$,
parity $i$, and dimension $\dim M = |\lambda| = n$.
Assume also that $\lambda$ has $k$ parts. Let 
$P$ denote the $\Lambda$-submodule of $M$ spanned by
the vectors $\big\{ v_{t,0} \, | \, 0 \leq t < k \big\}$
associated to the first column of $\lambda$.
Let $Q$ denote the span of the remaining basis vectors 
$\big\{ v_{s,t} \, | \, 1 \leq t < \lambda_s \big\}$.
Evidently both $P$ and the quotient module $M/P$ are shape
modules of parity $i$ and $1-i$ and shape $1^k$ and 
$\overline{\lambda}$ respectively where $\overline{\lambda}_t
:= \lambda_t - 1$ provided $\lambda_t > 0$ and zero
otherwise. 
Note $Q$ is a $\mathcal{A}$-submodule of $M$
where $\mathcal{A}$ is the $\Bbb{C}$-subalgebra of $\Lambda$
generated by $\alpha^*$, $\beta^*$, $e_0$, and $e_1$;
moreover $Q$ and $M/P$ as isomorphic 
when regarded as 
$\mathcal{A}$-modules.
Let $q: M \longrightarrow M/P$ and $\pr_{\s Q}: M \longrightarrow
Q$ denote the quotient and projection maps respectively.

\bigskip
\noindent
Let ${\bf c}$ be a bit string in $\{0,1\}^n_k$,
let ${\bf \overline{c}}$ denote the bit string
$(1 - c_1, \dots, 1-c_n)$ in $\{0,1\}^n_{n-k}$, and let
$\{ \bf d \}$ be an arbitrary bit string in $\{0,1\}^n$.
Let $\mathcal{F}^{\s \Lambda}_{\s \bf d} \big(M; {\bf c}\big)$ denote
the subset of $\mathcal{F}^{\s \Lambda}_{\s \bf d} \big(M\big)$
consisting of all $\Lambda$-composition
series $\big(M_t\big)_{\s n \geq t \geq 0}$ such that

\[ \dim {M_t \cap P \over { M_{t-1} \cap P}} \ = \ c_t \]

\bigskip
\noindent
whenever $n \geq t \geq 1$. Clearly this 
condition is equivalent to 

\[ \dim { \pr_{\s Q} M_t \over { \pr_{\s Q} M_{i-t} }  } \ = \ 1 - c_t \]

\bigskip
\noindent
whenever $n \geq t \geq 1$. An easy application of the
Zassenhaus butterfly lemma for vector spaces shows that

\[ \mathcal{F}^{\s \Lambda}_{\s \bf d}\big(M\big) \ = \
\bigsqcup_{{\bf c} \in \{0,1\}_k^n} \mathcal{F}^{\s \Lambda}_{\s \bf d}
\big(M; {\bf c}\big) \]

\begin{lemma} 
Let $M$ be a finite dimensional nilpotent left
$\Lambda$-module of $\Bbb{C}\big[\delta\big]$-partition
type $\lambda$.
For bit strings ${\bf c} \in \{0,1\}^n_k$ and 
${\bf d} \in \{0,1\}^n$

\[ \mathcal{F}^{\s \Lambda}_{\s \bf d} \big(M;{\bf c}\big) \ = \
\bigsqcup_{T} \ \Omega^{\s \Lambda
}_{\s T, \bf d} \big(M\big)\]

\bigskip
\noindent
where the union is taken over all standard tableaux $T$
of shape $\lambda$ whose first column content is 
${\bf [c]}$.

\end{lemma}

\begin{proof}
Note first that if $\big(M_t\big)_{\s n \geq t \geq 0}$ 
is a $\Lambda$-composition series
in $M$ then $\dim M_t \cap P$ equals the number of 
parts of the partition $\lambda^{(t)}$ associated to $M_t$
and hence equal to the number of rows in the
shape of $\lambda^{(t)}$. 
Recall that box is labeled $t$ 
in the standard tableau $T$ 
associated to $\big(M_t\big)_{\s n \geq t \geq 0}$
if and only if the shape of $\lambda^{(t)}$ contains the box
while the shape of $\lambda^{(t-1)}$ does not. 
Consequently a box in the first column of $T$ is
labeled $t$ if and only if the number of rows in
the shape of $\lambda^{(t)}$ is one more than
the number of rows in the shape of $\lambda^{(t-1)}$.
Equivalently a box in the first column of $T$
is labeled $t$ if and only if $\dim M_t \cap P
= 1 + \dim M_{t-1} \cap P$. Therefor $t$ labels
a box in the first column of $T$ if and only if
$c_t = 1$.

\end{proof}

\bigskip
\noindent 
It follows from lemma (2) that each $\mathcal{F}^{\s \Lambda}_{\s \bf d}
\big(M; {\bf c}\big)$ is a locally closed subvariety of
$\mathcal{F}^{\s \Lambda}_{\s \bf d}\big(M\big)$.
For bit strings ${\bf c} \in \{0,1\}^n_k$ and ${\bf d} \in \{0,1\}^n$ 
let $\pi_{\s \bf c}: \mathcal{F}^{\s \Lambda}_{\s \bf d}\big(M; {\bf c}\big)
\longrightarrow \mathcal{F}^{\s \Lambda}_{\s \bf c; d}
\big(P\big) \times \mathcal{F}^{\s \Lambda}_{\s \bf \overline{c}; d }
\big(Q\big)$ be the map given by

\[ \big(M_t\big)_{\s n \geq t \geq 0} \mapsto 
\big(M_t \cap P\big)_{\s n \geq t \geq 0} \times
\big(\pr_{\s Q} \big(M_t\big) \big)_{\s n \geq t \geq 0} \]

\bigskip
\noindent
where $\mathcal{F}^{\s \Lambda}_{\s \bf c; d} \big(P\big)$ 
denotes the variety of all {\it ${\bf c}$-step
$\Lambda$-composition series}, i.e.
$P_n := P \supseteq \cdots \supseteq P_0 := \{0\}$ in $P$ such that

\[ {P_t \over {P_{t-1}}} \ \simeq \ \left\{ \begin{array}{ll}
S_{d_t} &\text{if $c_t =1$} \\
\{ 0 \}&\text{if $c_t = 0$}
\end{array} \right. \]

\bigskip
\noindent
whenever $n \geq t \geq 1$.
Bearing a slight abuse of terminology and notation 
$\mathcal{F}^{\s \Lambda}_{\s \bf \overline{c}; d} \big(Q\big)$
denotes the variety of all ${\bf \overline{c}}$-step 
flags  $Q_n := Q \supseteq \cdots \supseteq Q_0:= \{0\}$ in $Q$
such that $q\big(Q_t\big)$ is a $\Lambda$-submodule
of $M/P$ and 

\[ { q\big(Q_t\big) \over{q\big(Q_{t-1}\big)}} \ \simeq \ 
\left\{ \begin{array}{ll} S_{d_t} &\text{if $c_t = 0$} \\
\{0\} &\text{if $c_t =1$}
\end{array} \right. \]

\bigskip
\noindent
whenever $n \geq t \geq 1$.

\bigskip
\noindent
Express the indicator sets of ${\bf c}$ and ${\bf \overline{c}}$ as
${\bf [c]} = \big\{r_1 < \dots < r_k \big\}$ and
${\bf [\overline{c}] } = \big\{ s_1 < \dots < s_{n-k} \big\}$
respectively; also set $r_0 = s_0 = 0$.
Set ${\bf e} = (e_1, \dots, e_k)$ with $e_t = d_{r_t}$
and ${\bf f} = (f_1, \dots, f_{n-k})$ with $f_t = d_{s_t}$.
Clearly $\mathcal{F}^{\s \Lambda}_{\s \bf c, d} 
\big(P\big)$ and $\mathcal{F}^{\s \Lambda}_{\s \bf e}
\big(P\big)$ are isomorphic as varieties, and so are 
$\mathcal{F}^{\s \Lambda}_{\s \bf \overline{c}, d} \big(Q\big)$
and $\mathcal{F}^{\s \Lambda}_{\s \bf f}\big(M/P\big)$. 
Indeed the isomorphisms $\psi_{\s P}$ and $\psi_{\s Q}$ 
in each case are given by

\[ \begin{array}{rlll} 
\mathcal{F}^{\s \Lambda}_{\s \bf c,d} \big(P\big) \ni 
&\Big(P_t\Big)_{\s n \geq t \geq 0}
&\stackrel{\psi_P}{\longmapsto} \Big(P_{r_t}\Big)_{\s k \geq t \geq 0} 
&\in \mathcal{F}^{\s \Lambda}_{\s \bf e} \big(P\big) \\ \\
\mathcal{F}^{\s \Lambda}_{\s \bf \overline{c}, d} \big(Q\big) \ni 
&\Big(Q_t\Big)_{\s n \geq t \geq 0}
&\stackrel{\psi_Q}{\longmapsto} \Big( q\big(Q_{s_t}\big) \Big)_{\s n-k \geq t \geq 0} 
&\in \mathcal{F}^{\s \Lambda}_{\s \bf f}\big(M/P\big)
\end{array} \]

\bigskip
\noindent
Any ${\bf c}$-step composition series 
$\big(P_i\big)_{\s n \geq i \geq 0}$ in 
$\mathcal{F}^{\s \Lambda}_{\s \bf c, d}\big(P\big)$
determines 
a ${\bf c}$-step flag of partitions
$1^{\s d_n} \supseteq \cdots \supseteq  1^{\s d_0}$ where $d_t = c_1 + \cdots 
+ c_t$ and $d_0 = 0$. This in turn corresponds to a
standard tableau $R$ of shape $1^k$ of
content ${\bf c}$. 
Let $\Omega^{\s \Lambda}_{\s R, \bf d}\big(P\big)$
denote the set of all ${\bf c}$-step composition series in
$\mathcal{F}^{\s \Lambda}_{\s \bf c, d}\big(P\big)$
associated to the standard tableau $R \in \Tab_{\bf c}\big(1^k\big)$.

\bigskip
\noindent
Similarly any ${\bf \overline{c}}$-step composition series
$\big(Q_t\big)_{\s n \geq t \geq 0}$ in
$\mathcal{F}^{\s \Lambda}_{\s \bf \overline{c}, d}
\big(Q\big)$ determines
a ${\bf \overline{c}}$-step flag of partitions $\overline{\lambda}^{(n)} \supseteq 
\cdots \supseteq \overline{\lambda}^{(0)}$ in $\overline{\lambda}$
where $\overline{\lambda}^{(t)}$
is the $\Bbb{C}\big[\delta\big]$-partition type
of the $\Lambda$-submodule $q\big(Q_t\big)$ of $M/P$.
Let $S$ be the standard tableau of shape $\overline{\lambda}$
and content ${\bf \overline{c}}$ corresponding to this 
${\bf \overline{c}}$-step flag.
Let $\Omega^{\s \Lambda}_{\s S, \bf d}\big(Q\big)$ denote
the set of ${\bf \overline{c}}$-step composition series in 
$\mathcal{F}^{\s \Lambda}_{\s \bf 1-c, d}\big(Q\big)$
associated to the standard tableau $S  \in \Tab_{\s \bf
\overline{c}}\big(\overline{\lambda}\big)$.

\bigskip
\noindent
Clearly 
$\psi_{\s P}\Big( \Omega^{\s \Lambda}_{\s R, \bf d}\big(P\big) \Big)
= \Omega^{\s \Lambda}_{\s \overline{R}, \bf e } \big( P \big)$ 
and $\psi_{\s Q}\Big(
\Omega^{\s \Lambda}_{\s S, \bf d}\big(Q\big) \Big)
= \Omega^{\s \Lambda}_{\s \overline{S}, \bf f}\big(M/P\big)$
where $\overline{R}$ and $\overline{S}$ 
are standard tableaux obtained from $R$ and $S$ by
replacing each index $r_t$ in $R$ by $t$ 
and each index $s_t$ in $S$ by $t$ respectively.

\begin{Prop}
Let $M$ be a shape module of shape $\lambda$
and parity $i$. Let $T$ be a standard tableau of 
shape $\lambda$ whose left-most column is labeled by
indices in ${\bf [c]}$.
Let $R \in \Tab_{\s \bf c}\big(1^k \big)$ and 
$S \in \Tab_{\s \bf \overline{c}}\big(\overline{\lambda} \, \big)$
be the pair of standard tableaux respectively obtained 
by taking the 1-st column and its complement in $T$, then

\[ \pi_{\s \bf c} \Big( \Omega^{\s \Lambda}_{\s T, \bf d}
\big(M\big) \Big) \ = \ \Omega^{\s \Lambda}_{\s R, \bf d}
\big(P\big) \times
\Omega^{\s \Lambda}_{\s S, \bf d}\big(Q\big) \]

\end{Prop}

\begin{proof}
Its is an immediate consequence of the definition
of $\Omega^{\s \Lambda}_{\s R, \bf d}\big(P\big)$
and $\Omega^{\s \Lambda}_{\s S, \bf d}\big(Q\big)$
that 

\[ \pi_{\s \bf c} \Big( \Omega^{\s \Lambda}_{\s T, \bf d}
\big(M\big) \Big) \ \subseteq \ \Omega^{\s \Lambda}_{\s R, \bf d}
\big(P\big) \times
\Omega^{\s \Lambda}_{\s S, \bf d}\big(Q\big) \]

\bigskip
\noindent
whenever $R$ and $S$ are respectively 
the left most column and its complement in a standard
tableau $T$ whose left-most column is labeled by
indices in ${\bf [c]}$.
Conversely suppose $\big(P_t\big)_{\s n \geq t \geq 0} \times
\big(Q_t\big)_{\s n \geq t \geq 0} \in 
\Omega^{\s \Lambda}_{\s R, \bf d}\big(P\big) \times
\Omega^{\s \Lambda}_{\s S, \bf d}\big(Q\big)$
where $R \in \Tab_{\s \bf c}\big(1^k\big)$ and $S \in 
\Tab_{\s \bf \overline{c}}\big(\overline{\lambda} \, \big)$.
Let $1^{\s d_n} \supseteq \cdots \supseteq
1^{\s d_0}$ and $\overline{\lambda}^{(n)} \supseteq
\cdots \supseteq \overline{\lambda}^{(0)}$ denote the corresponding
${\bf c}$-step and ${\bf \overline{c}}$-step 
flags of partitions where $d_t = c_1 + \cdots + c_t$ and
$d_0 = 0$. 
Note that $R$ and $S$ will be the left most column and complement in a
standard tableau $T$ whose left-most column is labeled by 
indices in ${\bf [c]}$ 
if and only if the number of parts of $\overline{\lambda}^{ (t)}$
is less than or equal to the number of parts of $1^{\s d_t}$
whenever $n \geq t \geq 0$. The latter condition holds

\bigskip
\indent \indent $\Longleftrightarrow$ \ $\dim \delta\big(Q_t\big) \cap P \leq
\dim P_t$ whenever $n \geq t \geq 0$

\indent \indent $\Longleftrightarrow$ \ $\delta\big(Q_t\big) \cap P \subseteq P_t$
whenever $n \geq t \geq 0$

\indent \indent $\Longleftrightarrow$ \ $\delta\Big(P_t \oplus Q_t\Big) 
\subseteq P_t \oplus Q_t$ whenever $n \geq t \geq 0$.

\bigskip
\noindent
Note that $Q_t$ must be a $\mathcal{A}$-submodule
of $M$ since $q\big(Q_t\big)$ is $\Lambda$-submodule of 
$M/P$ and because $Q$ itself is actually
a $\mathcal{A}$-submodule.  

\bigskip
\noindent
In addition $\delta\big(Q_t\big) \cap P$ contains both
$\alpha\big(Q_t\big) \cap P$ and $\beta\big(Q_t\big) \cap P$
whenever $n \geq t \geq 0$. To see this just note that 
we may express $v \in Q$ as $v = e_0 v + e_1 v$
and so $\alpha v = \delta e_0 v$ and $\beta v = \delta e_0 v$
where $e_0, e_1$ are the idempotents
of $\Lambda$.  
Consequently the third implication listed
above holds

\bigskip
\indent \indent $\Longleftrightarrow$ \ $\Big(P_t \oplus Q_t\Big)_{\s n \geq t \geq 0}$
is a $\Lambda$-composition series in $M$

\bigskip
\noindent
The proposition now follows from the observation 

\[ \pi_{\s \bf c}\Big(P_t \oplus Q_t\Big)_{\s n \geq t \geq 0} = 
\big(P_t\big)_{\s n \geq t \geq 0} \times
\big(Q_t\big)_{\s n \geq t \geq 0} \]

\end{proof}

\begin{Cor}
Let $M$ be a shape module of parity $i$
and shape $\lambda$
and let $T$ be any standard tableau of shape $\lambda$.
Either $\Omega^{\s \Lambda}_{\s T, \bf d}\big(M\big)$ is empty
or else its Euler characteristic 
$ \chi \Big( \Omega^{\s \Lambda}_{\s T,\bf d}\big(M\big) \Big)$ 
equals one. 

\end{Cor}

\begin{proof}
Let $\dim M = n$ and assume $\lambda$ has $k \geq 0$ parts.

\bigskip
\noindent
The corollary is clearly valid when $\lambda = 1^k$
because the partition admits only one standard
tableau and the corresponding shape module is uniserial
--- in which case $\Omega^{\s \Lambda}_{\s T, \bf d}\big(M\big)$
is either empty or else 
is a point and hence 
its Euler characteristic equals one.
Assume now that the number of columns of $\lambda$ is $N > 1$
and hypothesize inductively that the corollary holds for all
partitions possessing strictly fewer than $N$ columns. 

\bigskip
\noindent
Let $T$ be a
standard tableau of shape $\lambda$ and let ${\bf c}$ be
the unique bit string in $\{0,1\}^n_k$ corresponding to the 
content of the 1-st column of $T$; namely $c_t = 1$
if and only if $t$ labels a box in the first column of
$T$. 

\bigskip
\noindent
By proposition 1 above $\pi_{\s \bf c}$ maps 
$\Omega^{\s \Lambda}_{\s T, \bf d}\big(M\big)$
onto $\Omega^{\s \Lambda}_{\s R, \bf d}\big(P\big)
\times \Omega^{\s \Lambda}_{\s S, \bf d}\big(Q\big)$
where $R \in \Tab_{\s \bf c}\big(1^k\big)$ and 
$S \in \Tab_{\s \bf \overline{c}}\big(\overline{\lambda})$ are 
the tableaux obtained by taking the 1-st column and its complement
in $T$. 
As mentioned before $\Omega^{\s \Lambda}_{\s R,\bf d}\big(P\Big)
\simeq \Omega^{\s \Lambda}_{\s \overline{R}, \bf e} \big(P\big)$
and $\Omega^{\s \Lambda}_{\s S, \bf d}\big(Q\big)
\simeq \Omega^{\s \Lambda}_{\s \overline{S}, \bf f}\big(M/P\big)$
where $\overline{R}$ and $\overline{S}$ are standard tableaux
of shapes $1^k$ and $\overline{\lambda}$ respectively. Since
both $P$ and $M/P$ are shape modules whose associated 
partitions have shapes with strictly less than $N$ columns 
then we may inductively conclude that both 
$\Omega^{\s \Lambda}_{\s R, \bf d}\big(P\big)$ and
$\Omega^{\s \Lambda}_{\s S, \bf d}\big(Q\big)$ 
are either empty or have Euler characteristic one.

\bigskip
\noindent
Clearly $\Omega^{\s \Lambda}_{\s T, \bf d}\big(M\big)$
is empty if and only if 
either $\Omega^{\s \Lambda}_{\s R, \bf d}\big(P\big)$ 
is empty or
$\Omega^{\s \Lambda}_{\s S, \bf d}\big(Q\big)$ is empty.
Let us suppose that both
$\Omega^{\s \Lambda}_{\s R, \bf d}\big(P\big)$ and
$\Omega^{\s \Lambda}_{\s S, \bf d}\big(Q\big)$ are
non-empty and thus have Euler characteristic equal to one.

\bigskip
\noindent
By adapting lemmas 3.1.1 and 3.2.2 of 
\cite{GLS} it follows that a composition series 
$\big(M_t\big)_{\s n \geq t \geq 0}$ inside
$\Omega^{\s \Lambda}_{\s T, \bf d}\big(M\big)$
will be a pre-image with respect to $\pi_{\s \bf c}$  
of a pair $\big(P_t\big)_{\s n \geq t \geq 0} \times
\big(Q_t\big)_{\s n \geq t \geq 0}$ in 
 $\Omega^{\s \Lambda}_{\s R, \bf d}\big(P\big)
\times \Omega^{\s \Lambda}_{\s S, \bf d}\big(Q\big)$
if and only if there exists a linear map
$\theta: M \longrightarrow P$ 
satisfying

\bigskip
\indent \indent $\bullet$ $P \subseteq \ker \theta$ 
and $\big[ e_i , \theta \big] = 0$ for $i=0,1$

\bigskip
\indent \indent $\bullet$ 
$\big[\delta , \theta\big]\big(Q_t\big) \subseteq P_t$
and $\big[\delta^*, \theta \big]\big(Q_t \big) \subseteq P_t$
whenever $n \geq t \geq 1$ where $\delta^* = \alpha^* + \beta^*$

\bigskip
\noindent
such that $M_t = P_t \oplus_{\s \theta}
Q_t$ whenever $n \geq t \geq 0$ where

\[ P_t \oplus_{\s \theta} Q_t \ := P_t \oplus 
\Big\{ \, \theta(x) + x \, \Big| \, x \in Q_t \, \Big\}. \]

\bigskip
\noindent
Note that the commutator identities are necessary
and sufficient conditions to insure that 
$\big(P_t \oplus_{\s \theta} Q_t \big)_{\s n \geq t \geq 0}$
is a $\Lambda$-composition series
in $M$. There is a degree of redundancy in this description
of the $\pi_{\s \bf c}$-preimages since 

\[ P_t \oplus_{\s \theta} Q_t \ = \ P_t \oplus_{\s \zeta} Q_t \]

\bigskip
\noindent
if and only if $(\theta - \zeta)\big(Q_t\big) \subseteq P_t$
whenever $n \geq t \geq 0$. In this case we declare $\theta$
and $\zeta$ to be equivalent and write 
$\theta \sim \zeta$. As the authors of \cite{GLS} point out 
in lemma 3.2.2 the kernel and commutator constraints and the equivalence
relation $\sim $ are all linear conditions and
thus the fiber of preimages under $\pi_{\s \bf c}$ 
must be affine.

\bigskip
\noindent
The map $\pi_{\s \bf c}$ is known to be a morphism 
of varieties and thus it descends to a morphism between
$\Omega^{\s \Lambda}_{\s T, \bf d}\big(M\big)$
and $\Omega^{\s \Lambda}_{\s R, \bf d}\big(P\big)
\times \Omega^{\s \Lambda}_{\s S, \bf d}\big(Q\big)$.
Moreover the Euler characteristic of any of its 
fibers is equal to one owing to the fact that
each fiber is isomorphic to an affine space.
By proposition 7.4.1 in \cite{GLS} we may
conclude in this case that 

\[ \chi \Big( \, \Omega^{\s \Lambda}_{\s T, \bf d}\big(M\big) \, \Big)
\ = \ \chi \Big( \,
\Omega^{\s \Lambda}_{\s R, \bf d}\big(P\big)
\times \Omega^{\s \Lambda}_{\s S, \bf d}\big(Q\big) \, \Big) \]

\bigskip
\noindent
On the other hand Euler characteristic is multiplicative 
therefore

\[ \begin{array}{ll}
\chi \Big( \, \Omega^{\s \Lambda}_{\s T, \bf d}\big(M\big) \, \Big)
&= \ \chi \Big( \,
\Omega^{\s \Lambda}_{\s R, \bf d}\big(P\big)
\times \Omega^{\s \Lambda}_{\s S, \bf d}\big(Q\big) \, \Big) \\ 
&= \ \chi \Big( \, \Omega^{\s \Lambda}_{\s R, \bf d}\big(P\big) \, \Big)
\, \cdot \, \chi \Big( \, 
\Omega^{\s \Lambda}_{\s S, \bf d}\big(Q\big) \, \Big) \\
&= \ 1 \end{array} \]

\end{proof}

\begin{Cor}
Let $M$ be a shape module of parity $i$, shape $\lambda$, 
and dimension $n$. Let ${\bf d}$ be a bit string
in $\{0,1\}^n$ and let $T$ be a standard
tableau of shape $\lambda$. Then $\Omega^{\s \Lambda}_{\s T, \bf d}
\big(M\big)$ will be non-empty if and only if 
the $i$-parity of $T$ equals ${\bf d}$.
\end{Cor}

\begin{proof}
Assume $\lambda$ has $k$ parts. The corollary is clearly
valid when $\lambda = 1^k$ because the partition admits only
on standard tableau and the corresponding shape module is
uniserial. In this case $\Omega^{\s \Lambda}_{\s T, \bf d}
\big(M\big)$ will be non-empty if and only if 
$d_t$ equals the parity $t + i$ --- which is 
precisely the $i$-parity of the box in row $t$.
Assume now that the number of columns of $\lambda$
is $N > 1$ and hypothesize inductively that the corollary
holds for partitions with strictly fewer than $N$ columns.

\bigskip
\noindent
Let $T$ be a
standard tableau of shape $\lambda$ and let ${\bf c}$ be
the unique bit string in $\{0,1\}^n_k$ corresponding to the 
content of the 1-st column of $T$; namely $c_t = 1$
if and only if $t$ labels a box in the first column of
$T$. Write ${\bf [c]} = \big\{ r_1, \dots, r_k \big\}$
and ${\bf [\overline{c}]} = \big\{ s_1, \dots, s_{n-k} \big\}$.

\bigskip
\noindent
By proposition 1 we know 
$\Omega^{\s \Lambda}_{\s T, \bf d}\big(M\big)$
is non-empty if and only if 
both $\Omega^{\s \Lambda}_{\s R, \bf d}\big(P\big)$ and
$\Omega^{\s \Lambda}_{\s S, \bf d}\big(Q\big)$ are 
non-empty where $R \in \Tab_{\s \bf c}\big(1^k\big)$ 
and $S \in \Tab_{\s \bf \overline{c}}\big(\overline{\lambda}\big)$ 
are the tableaux obtained by taking the 1-st column and its complement
in $T$. 
On the other hand $\Omega^{\s \Lambda}_{\s R, \bf d}\big(P\big)$ and
$\Omega^{\s \Lambda}_{\s S, \bf d}\big(Q\big)$ are 
non-empty if and only if    
$\Omega^{\s \Lambda}_{\s \overline{R}, \bf e} \big(P\big)$
and $\Omega^{\s \Lambda}_{\s \overline{S}, \bf f}\big(M/P\big)$
are respectively non-empty; recall that $\overline{R}$ and
$\overline{S}$ are the standard tableaux obtained from
$R$ and $S$ by replacing the entries 
$r_t$ and $s_t$ by $t$ respectively. Also
$e_t = d_{r_t}$ whenever $k \geq t \geq 1$ and 
$f_t = d_{s_t}$ whenever $n-k \geq t \geq 1$.

\bigskip
\noindent
Since both $P$ and $M/P$ are shape modules of parities $i$ and
$1-i$ and since their respective shapes $1^k$ and $\overline{\lambda}$ 
have fewer than $N$ columns we may apply the inductive assumption and 
conclude that 

\[ \begin{array}{ll}
\text{ $\Omega^{\s \Lambda}_{\s \overline{R}, \bf e}$ non-empty} 
&\Longleftrightarrow \ \text{$e_t$ equals $i$-parity of box $t$ in 
$\overline{R}$} \\
\text{ $\Omega^{\s \Lambda}_{\s \overline{S}, \bf f}$ non-empty} 
&\Longleftrightarrow \ \text{$f_t$ equals $(1-i)$-parity of box $t$ in 
$\overline{S}$} \end{array} \]

\bigskip
\noindent
equivalently

\[ \begin{array}{ll}
\text{ $\Omega^{\s \Lambda}_{\s R, \bf d}$ non-empty} 
&\Longleftrightarrow \ \text{$d_{r_t}$ equals $i$-parity of box $r_t$ in 
$R$} \\
\text{ $\Omega^{\s \Lambda}_{\s S, \bf d}$ non-empty} 
&\Longleftrightarrow \ \text{$d_{s_t}$ equals $i$-parity of box $s_t$ in 
$S$} \end{array} \]

\bigskip
\noindent
Bear in mind that the $(1-i)$-parity of a box in $\overline{S}$
is equal to the $i$-parity of the same box in $S$ owing to the fact
that first column in $S$, which is situated to the immediate right
of the first column in $T$, is counted as column $1$ not $0$.
The corollary follows given that 
$[1 \dots n] = \big\{d_{r_t} \, | \, k \geq t \geq 1 \big\}
\sqcup \big\{ d_{s_t} \, | \, n-k \geq t \geq 1 \big\}$.

\end{proof}

\bigskip
\noindent
{\bf Proof of Theorem 1:}
By lemma (1) we know that

\[ \chi \Big( \mathcal{F}^{\s \Lambda}_{\s \bf d}\big(M\big) \Big)
\ = \ \sum_{T \in \Tab(\lambda)} \chi \Big( \,
\Omega^{\s \Lambda}_{\s T, \bf d}\big(M\big) \,
\Big) \]

\bigskip
\noindent
By Corollary (1) each Euler characteristic $\chi \Big( \,
\Omega^{\s \Lambda}{\s T, \bf d}\big(M\big) \, \Big)$
contributes either $1$ or $0$ and  
Corollary (2) stipulates that only 
standard tableau $T$ whose $i$-parity equals
${\bf d}$ add a non-zero contribution. Therefore 
the Euler characteristic $\chi \Big( \, 
\mathcal{F}^{\s \Lambda}_{\s \bf d}\big(M\big) \, \Big)$
equals the number of standard tableaux of shape $\lambda$
whose $i$-parity equals ${\bf d}$. 

\bigskip
\noindent
Taking into account Proposition (1) from the first section 
we can furthermore conclude:

\begin{Cor}
Let $M$ be a shape module of shape $\lambda$,
parity $i$, and dimension $n$. 
Let ${\bf d}$ be a bit string $\{0,1\}^n$ which
can be expressed as ${\bf i}^{\bf j}$ for some
alternating bit string ${\bf i}$ in $\{0,1\}^k$
and non-negative integer $k$-tuple ${\bf j}$
in $\Bbb{Z}^k_{\s \geq 0}$, then the
Euler characteristic 
$\chi \Big( \, \mathcal{F}^{\s \Lambda}_{\s \bf d}\big(M\big)
\, \Big) $ equals 

\[  j_1! \cdots j_k! \
\Big| \chess^{(i^*)}_{\s \bf j} \big(\lambda \big) \, \Big| \]

\bigskip
\noindent
where $i^*$ is the parity of  $i + i_1 + 1$.

\end{Cor}

\begin{Conj}
Let $\lambda$ be a partition of $n$, let $i = 0,1$ a choice 
of parity, and let ${\bf d}$ be a bit string in $\{0,1\}^n$.
Let $M$ be a shape module of shape $\lambda$ and parity $i$
and let $T \in \super^{(i)} \big(\lambda: {\bf d})$, then
$\Omega^{\Lambda}_{\s \bf d,T}(M)$ is an affine space of 
dimension $\gr(T)$. Alternatively, if 
$q$ denotes a power of a prime and $\Bbb{F}_q$ is a finite
field with $q$ elements then the number of $\Bbb{F}_q$-rational
points of $\mathcal{F}^{\s \Lambda}_{\bf d}\big(M\big)$ is

\[ \sum_{T \in \super^{(i)}(\lambda;
{\bf d})}    q^{ \gr(T) } \]

\bigskip
\noindent
where, accordinging to definition (5) of section (1),
$\gr(T)$ is the number of grounded ${\bf d}$-transposition
pairs of $T$.

\end{Conj}

\bigskip
\section{Block-Toeplitz Representation of $\SL_2(\mathcal{L})$:}

\bigskip
\bigskip
\noindent
In this section we recount a standard realization of 
the loop group in terms of infinite block-Toeplitz matrices
described in lecture 9 of \cite{Kac}.

\bigskip
\noindent
Consider the variable $t$ as a coordinate for the circle
subgroup $S^1$ in $\Bbb{C}^*$ and 
let $\mathcal{H} := L^2\big(S^1; \Bbb{C}^2\big)$ be the Hilbert space of all
square integrable vector-valued functions $f:S^1 \longrightarrow
\Bbb{C}^2$ expressed in the coordinate $t$.  
A loop $g = \big(g_{ij}\big)$ in
$\SL_2(\mathcal{L})$
gives rise to a multiplication operator $T_g:\mathcal{H}
\longrightarrow \mathcal{H}$ defined by 

\[  T_g \begin{pmatrix} f_1 \\ f_2 \end{pmatrix} \ := \ 
\begin{pmatrix} g_{11} & g_{12} \\ g_{21} & g_{22} \end{pmatrix}
\cdot \begin{pmatrix} f_1 \\ f_2 \end{pmatrix} \]

\bigskip
\noindent
Note: the component functions of the result of this matrix
multiplication are still square integrable since, over a compact
domain such as $S^1$, the product
of a bounded integrable function (in this case a polynomial function) 
and a square integrable function
remains square integrable. 

\bigskip
\noindent
The map $T_g$ is clearly linear and so we may write down
the matrix representing $T_g$ 
with respect to the (ordered) Fourier basis 
$ \big\{ \psi_{n} \ \big| \ n \in \Bbb{Z} \big\}$
of $\mathcal{H}$ where

\[ \psi_{2j + i} := \ t^{-j} \, \vec{{\bf e}}_i \ \ \text{for 
$j \in \Bbb{Z}$ and $i = 1,2$}. \]

\bigskip
\noindent
The matrix of $T_g$ will be the $\Bbb{Z} \times \Bbb{Z}$ matrix 
whose $(M,N)$ entry is
 
\[ \Res {\D {g_{ij} \over {t^{n-m+1} } }} \]

\bigskip
\noindent
where $M = 2m + i$ and $N = 2n + j$ with
$i, j \in \{1, 2\}$ and where $\Res$ means
residue. Alternatively $T_g$ can be 
expressed in block form

\[ \begin{pmatrix} \ddots & & & & \\ & a_0 & a_1 & a_2 &  \\
& a_{-1} & a_0 & a_1 &  \\ & a_{-2} & a_{-1} & a_0  \\ & & & &  \ddots
\end{pmatrix} \]

\bigskip
\noindent
where $a_k := \bigg(\Res \D {g_{ij} \over {t^{k+1}} } \bigg)$
for $k \in \Bbb{Z}$.
The $2 \times 2$ matrix $a_k$ is precisely the  
coefficient matrix of $t_k$ in the Fourier expansion
$g = {\D \sum_{k \in \Bbb{Z}} a_k t^k}$.

\bigskip
\noindent
Such a $\Bbb{Z} \times \Bbb{Z}$ matrix in block form with constant block diagonals
will be called a {\it block Toeplitz matrix}. We will identify $T_g$ with its matrix
and the map $g \mapsto T_g$ defines an injective homomorphism, denoted $T$, from the loop group
$\SL_2(\mathcal{L})$ to the {\it restricted general linear group}
$\GL_{\res}(\Bbb{C})$ (see \cite{Segal} chapter 6).

\begin{Def}
Let $\mu \subseteq \lambda$ be a pair of ordered partitions, let $i = 0,1$
be a choice of parity, and set $N := \max \big(N_\mu, N_\lambda \big)$.
For an element $g$ in the loop group
$\SL_2(\mathcal{L})$ define $\Delta^{(i)}_{\mu, \lambda}\big(g \big)$
to be the determinant of the
$N \times N $ submatrix of $T_g$ whose row
and columns sets are 
$\set^{(i)}_{N_\lambda}(\mu)$ and 
$\set^{(i)}_{N_\lambda}(\lambda)$ respectively.
\end{Def}

\begin{Rem}
For each $g$ the minor $\Delta^{(i)}_{\mu, \lambda}$ is 
a polynomial
in the matrix entries of $T_g$ and as such
is a regular function over $\SL_2(\mathcal{L})$.
\end{Rem}

\begin{Rem}
For a non-negative integer $n$ 
let $E^{(i)}_n$ be a short hand notation
for the minor $\Delta^{(i)}_{\s \emptyset, \lambda}$
where $\lambda$ is the partition with
$\lambda_0 = n$ and $\lambda_k = 0$ whenever
$k > 0$. 
Evidently $E^{(i)}_n$ is the $\big(i,n+i\big)$-entry
of $T_g$. In view of the fact that
the $(M,N)$ and $(M+2,N+2)$ entries of $T_g$ are equal
one observes that the following {\bf $i$-Pieri} rule must hold
for $g \in U_+$:

\[ \Delta^{(i)}_{{\s \emptyset}, \lambda} \ = \
\det \Bigg(
{\mathlarger{\mathlarger{\mathlarger E}}}^{\D (q_s)}_{\D p_{st}} 
\Bigg) \]

\bigskip
\noindent
where $N = N_\lambda$ and $(q_0, \dots, q_{\s N})$ is the
alternating bit string in $\{0,1\}^{\s N}$ starting 
with $q_0 = i$ and $p_{st} = \lambda_{{\s N} - s} + t - s$ 
whenever $N \geq s,t \geq 0$. Here we use the convention
that $E^{(i)}_n = 0$ whenever $n$ is negative.

\end{Rem}

\begin{Def}
For non-negative integers $m$ and $n$ let 
$\Delta^{(i)}_{m,n}$ be short hand notation for 
$\Delta^{(i)}_{\mu, \lambda}$ where
$\mu$ is the partition with 
$\mu_k = m - k$ whenever $m \geq k$ and $\mu_k = 0$ whenever
$k > m$ and where $\lambda$ is the partition 
with $\lambda_k = n -k$
whenever $n \geq k$ and $\lambda_k = 0$ whenever
$k > n$.
\end{Def}

\begin{Rem}
Up to re-indexing and a {\it twist} as defined in \cite{BIRS}
the minors $\Delta^{(i)}_{0, n}$
are the minors considered in conjecture 4.3 of \cite{BIRS}.
\end{Rem}

\begin{Conj}
The minors $\Delta^{(i)}_{m,n}$ are precisely the 
{\bf generalized minors} of 
Fomin-Zelevinsky defined in \cite{Bruhat}.
\end{Conj}

\bigskip
\section{Path Formalism and Resolution:}

\bigskip
\bigskip
\noindent
This section describes a locally finite but nevertheless
infinite version of the Fomin-Zelevinsky
graphical calculus which was originally developed
in \cite{Positivity} and \cite{Bruhat} as
a means to parameterize the double
Bruhat cells of a simply connected 
simple algebraic group.

\bigskip
\noindent
For $a \in \Bbb{C}^*$ and a choice of 
parity $i = 0, 1$
the {\it chip diagram}
$\Gamma_i(a)$ is
is a weighted directed planar
graph whose vertices are
either sources or sinks:
The set of  
sources and the set of sinks 
are both indexed by $\Bbb{Z}$ 
obeying the rule that $source$ $n$ 
is connected to sink $m$ if
and only if either $n = m$ or
else $m = n+1$ and the parity
of $n$ is $i$. Finite portions of the chip diagrams
$\Gamma_0(a)$ and $\Gamma_1(a)$
are depicted below:

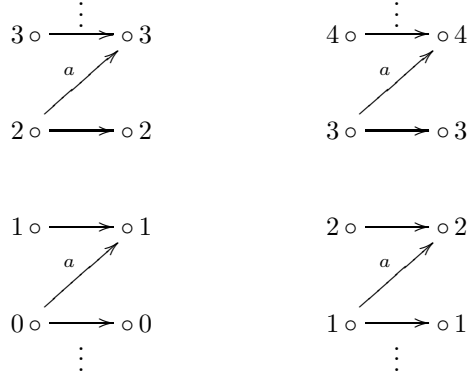
\[ 
\xymatrix{ 
3 \, \circ \ar[r]^{\vdots} & \circ \, 3 \\
2 \, \circ \ar[r] \ar[ur]^{a} & \circ \, 2 \\ 
1 \, \circ \ar[r] & \circ \, 1 \\
0 \, \circ \ar[r]_{\vdots} \ar[ur]^{a} & \circ \, 0  } 
\qquad \qquad \qquad
\xymatrix{ 
4 \, \circ \ar[r]^{\vdots} & \circ \, 4 \\
3 \, \circ \ar[r] \ar[ur]^{a} & \circ \, 3 \\ 
2 \, \circ \ar[r] & \circ \, 2 \\
1 \, \circ \ar[r]_{\vdots} \ar[ur]^{a} & \circ \, 1 }
\] 

\bigskip
\noindent
The diagonal edges carry weight $a$
while all other edges 
are assumed to carry weight $1$. 
Given an alternating bit string
${\bf i} = (i_1, \dots, i_k)$ 
in $\{0,1\}^k$
and a $k$-tuple
of parameters ${\bf a} = (a_1, \dots, a_k)$
in $\big( \Bbb{C}^* \big)^k$ let
$\Gamma_{i}( {\bf a} )$
denote the graph obtained by
{\it concatenating} 
the chip diagrams $\Gamma_{i_1}(a_1),
\dots, \Gamma_{i_k}(a_k)$
from left to right starting with $\Gamma_{i_1}(a_1)$.
To concatenate $\Gamma_{i_s}(a_s)$ and
$\Gamma_{i_{s+1}}(a_{s+1})$ simply
graft each sink of $\Gamma_{i_s}(a_s)$ with the 
source of $\Gamma_{i_{s+1}}(a_{s+1})$ having the
same integer label. The set of sources and sinks of the
graph $\Gamma_{\bf i}({\bf a})$ are both 
indexed by $\Bbb{Z}$ which will be respectively depicted
on the left and right.  For example
$\Gamma_{\s (1,0,1,0)}\big(a_1,a_2,a_3,a_4\big)$ is shown
here:

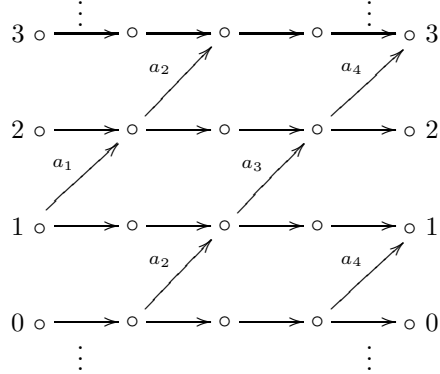
\[ \xymatrix{ 
3 \ \circ  \ar[r]^{\vdots} 
& \circ \ar[r] 
& \circ \ar[r] 
& \circ \ar[r]^{\vdots} 
& \circ \ 3 
\\
2 \ \circ \ar[r]
& \circ \ar[r] \ar[ur]^{a_2} 
& \circ \ar[r] 
& \circ \ar[r] \ar[ur]^{a_4} 
& \circ \ 2 
\\
1 \ \circ \ar[r] \ar[ur]^{a_1}
& \circ \ar[r] 
& \circ \ar[r] \ar[ur]^{a_3}
& \circ \ar[r] 
& \circ \ 1 
\\
0 \ \circ \ar[r]_{\vdots} 
& \circ \ar[r] \ar[ur]^{a_2}
& \circ \ar[r]
& \circ \ar[r]_{\vdots} \ar[ur]^{a_4}
& \circ \ 0  } \]

\bigskip
\noindent
The {\it weight matrix} $x\Big(\Gamma_{\bf i}({\bf a}) \Big)$ 
is the $\Bbb{Z} \times \Bbb{Z}$ matrix
whose $i,j$ entry is the sum of all weights $\wt(\pi)$ of paths
$\pi$ joining the $i$-th source to the $j$-th sink in
$\Gamma_{\bf i}({\bf a})$. By definition
the weight $\wt(\pi)$ of a path $\pi$ is the product of the
weights of edges comprising the path $\pi$.

\bigskip
\noindent
Note that the $\Bbb{Z} \times \Bbb{Z}$ 
weight matrices
$x\Big(\Gamma_i(a)\Big)$ are exactly
the block Toeplitz matrices 
$T_i(a) := T_{x_i(a)}$ 
associated with 
the element $x_i(a)$ in the 
loop group $\SL_2\big(\mathcal{L})$
for $i=0,1$ and $a \in \Bbb{C}$.
Weight matrices are multiplicative
in the sense that 

\[ x\Big( \Gamma_{\bf i}({\bf a}) \Big)
\ = \
x\Big( \Gamma_{i_1}(a_1) \Big) \cdots
x\Big( \Gamma_{i_k}(a_k) \Big) \]

\bigskip
\noindent
and therefore 

\begin{equation}
x \Big( \Gamma_{\bf i} ( {\bf a} ) \Big)
= T_{i_1}(a_1) \cdots T_{i_k}(a_k). \end{equation}

\bigskip
\noindent
Given two finite subsets $U= \{u_0, \dots, u_{\s N} \}$ and
$V = \{v_0, \dots, v_{\s N} \}$ of $\Bbb{Z}$  
the famous {\it Lindstr\"om lemma} (see \cite{Positivity} 
lemma 1) asserts in the present context that the matrix minor 
$\Delta_{U,V}$ of the weight matrix 
$x\Big( \Gamma_{{\bf i}}({\bf a})
\Big)$ is given by the sum

\begin{equation}
\Delta_{U,V} \ = \
\sum_{ \substack{  {\bf \pi} =  \{\pi_0, \dots, \pi_N \} \\ 
u_n \xrightarrow{\pi_n} v_n \\ \text{non-crossing}  }}
\ \wt(\pi_0) \cdots \wt(\pi_{\s N}) \end{equation}

\bigskip
\noindent
where the sum is taken over all families of paths 
${\bf \pi} = \{ \pi_0, \dots, \pi_{\s N} \}$ in 
$\Gamma_{{\bf i}}({\bf a})$
whose members are pairwise non-crossing (i.e.
sharing no vertices or edges) and $\pi_n$ joins
source $u_n$ to sink $v_n$ whenever $N \geq n \geq 0$.

\bigskip
\noindent
As an illustration consider the case of ${\bf i}
= (1,0,1,0)$ with $U = \{ 0, 1 \}$ and
$V= \{ 1, 3\}$. The corresponding 
families of non-crossing paths ${\bf \pi} =
\{ \pi_0, \pi_1 \}$ in $\Gamma_{{\bf i}}({\bf a})$
are depicted below (in bold font):

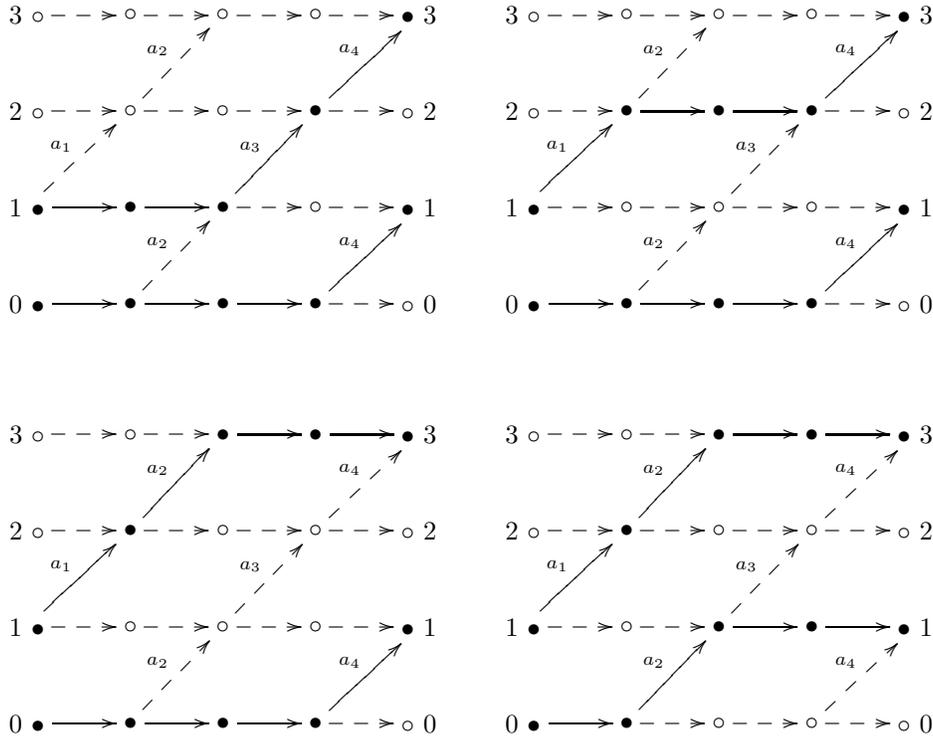
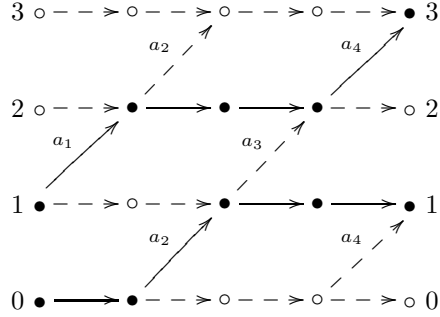
\[ \xymatrix { 
3 \ \circ \ar@{-->}[r] 
& \circ \ar@{-->}[r] 
& \circ \ar@{-->}[r] 
& \circ \ar@{-->}[r] 
& \bullet \ 3 
\\
2 \ \circ \ar@{-->}[r]
& \circ \ar@{-->}[r] \ar@{-->}[ur]^{a_2} 
& \circ \ar@{-->}[r] 
& \bullet \ar@{-->}[r] \ar[ur]^{a_4} 
& \circ \ 2 
\\
1 \ \bullet \ar[r] \ar@{-->}[ur]^{a_1}
& \bullet \ar[r] 
& \bullet \ar@{-->}[r] \ar[ur]^{a_3}
& \circ \ar@{-->}[r] 
& \bullet \ 1 
\\
0 \ \bullet \ar[r] 
& \bullet \ar[r] \ar@{-->}[ur]^{a_2}
& \bullet \ar[r]
& \bullet \ar@{-->}[r] \ar[ur]^{a_4}
& \circ \ 0 } 
\qquad 
\xymatrix{ 
3 \ \circ \ar@{-->}[r] 
& \circ \ar@{-->}[r] 
& \circ \ar@{-->}[r] 
& \circ \ar@{-->}[r] 
& \bullet \ 3 
\\
2 \ \circ \ar@{-->}[r]
& \bullet \ar[r] \ar@{-->}[ur]^{a_2} 
& \bullet \ar[r] 
& \bullet \ar@{-->}[r] \ar[ur]^{a_4} 
& \circ \ 2 
\\
1 \ \bullet \ar@{-->}[r] \ar[ur]^{a_1}
& \circ \ar@{-->}[r] 
& \circ \ar@{-->}[r] \ar@{-->}[ur]^{a_3}
& \circ \ar@{-->}[r] 
& \bullet \ 1 
\\
0 \ \bullet \ar[r] 
& \bullet \ar[r] \ar@{-->}[ur]^{a_2}
& \bullet \ar[r]
& \bullet \ar@{-->}[r] \ar[ur]^{a_4}
& \circ \ 0 } 
\]

\bigskip
\[ \xymatrix{ 
3 \ \circ \ar@{-->}[r] 
& \circ \ar@{-->}[r] 
& \bullet \ar[r] 
& \bullet \ar[r] 
& \bullet \ 3 
\\
2 \ \circ \ar@{-->}[r]
& \bullet \ar@{-->}[r] \ar[ur]^{a_2} 
& \circ \ar@{-->}[r] 
& \circ \ar@{-->}[r] \ar@{-->}[ur]^{a_4} 
& \circ \ 2 
\\
1 \ \bullet \ar@{-->}[r] \ar[ur]^{a_1}
& \circ \ar@{-->}[r] 
& \circ \ar@{-->}[r] \ar@{-->}[ur]^{a_3}
& \circ \ar@{-->}[r] 
& \bullet \ 1 
\\
0 \ \bullet \ar[r] 
& \bullet \ar[r] \ar@{-->}[ur]^{a_2}
& \bullet \ar[r]
& \bullet \ar@{-->}[r] \ar[ur]^{a_4}
& \circ \ 0 } 
\qquad
\xymatrix{ 
3 \ \circ \ar@{-->}[r] 
& \circ \ar@{-->}[r] 
& \bullet \ar[r] 
& \bullet \ar[r] 
& \bullet \ 3 
\\
2 \ \circ \ar@{-->}[r]
& \bullet \ar@{-->}[r] \ar[ur]^{a_2} 
& \circ \ar@{-->}[r] 
& \circ \ar@{-->}[r] \ar@{-->}[ur]^{a_4} 
& \circ \ 2 
\\
1 \ \bullet \ar@{-->}[r] \ar[ur]^{a_1}
& \circ \ar@{-->}[r] 
& \bullet \ar[r] \ar@{-->}[ur]^{a_3}
& \bullet \ar[r] 
& \bullet \ 1 
\\
0 \ \bullet \ar[r] 
& \bullet \ar@{-->}[r] \ar[ur]^{a_2}
& \circ \ar@{-->}[r]
& \circ \ar@{-->}[r] \ar@{-->}[ur]^{a_4}
& \circ \ 0 } 
\]

\bigskip
\[ \xymatrix{ 
3 \ \circ \ar@{-->}[r] 
& \circ \ar@{-->}[r] 
& \circ \ar@{-->}[r] 
& \circ \ar@{-->}[r] 
& \bullet \ 3 
\\
2 \ \circ \ar@{-->}[r]
& \bullet \ar[r] \ar@{-->}[ur]^{a_2} 
& \bullet \ar[r] 
& \bullet \ar@{-->}[r] \ar[ur]^{a_4} 
& \circ \ 2 
\\
1 \ \bullet \ar@{-->}[r] \ar[ur]^{a_1}
& \circ \ar@{-->}[r] 
& \bullet \ar[r] \ar@{-->}[ur]^{a_3}
& \bullet \ar[r] 
& \bullet \ 1 
\\
0 \ \bullet \ar[r] 
& \bullet \ar@{-->}[r] \ar[ur]^{a_2}
& \circ \ar@{-->}[r]
& \circ \ar@{-->}[r] \ar@{-->}[ur]^{a_4}
& \circ \ 0 } \]

\bigskip
\noindent
which agrees with the manual computation of the
matrix minor shown here

\[ \begin{array}{ll}
&\Delta_{U,V} \Bigg( x \Big( \Gamma_{\bf i}({\bf a}) \Big) \Bigg) \\
&= \ \Delta_{U,V} \Bigg( x\Big( \Gamma_1(a_1) \Big)
\cdot x\Big( \Gamma_0(a_2)\Big) \cdot  x\Big( \Gamma_1(a_3)\Big)
\cdot x\Big( \Gamma_0(a_4)\Big) \Bigg) \\
&= \ \Delta_{U,V} \Bigg( T_1(a_1) \cdot T_0(a_2) \cdot T_1(a_3) 
\cdot T_0(a_4) \Bigg) \\ \\
&= \ a_3a_4^2 + a_1a_4^2 + a_1 a_2^2
+ 2a_1a_2a_4 \end{array} \]

\begin{Def} 
Let $\mu \subseteq \lambda$ be a pair of ordered 
partitions, set $N = N_\lambda$, and let
$i = 0,1$ be a choice of parity. In addition
let ${\bf i}=(i_1, \dots, i_k)$ be an alternating
bit string in $\{0,1\}^k$ and let ${\bf a} = (a_1, \dots, a_k)
\in \big(\Bbb{C}^*\big)^k$.
Define $\Path^{(i)}_{\, \bf i} \big(\mu,\lambda\big)$
to be the collection of all families of 
non-crossing paths 
${\bf \pi} = (\pi_1, \dots, \pi_{\s N})$ 
in $\Gamma_{\bf i}({\bf a})$
which join sources in $U = \set^{(i)}_N(\mu)$
to sinks in $V = \set^{(i)}_N(\lambda)$. 
For ${\bf j} \in \Bbb{Z}_{\s \geq 0}$ let
$\Path^{(i)}_{\, \bf i} \big(\mu, \lambda; {\bf j} \big)$
denote those families of non-crossing paths
${\bf \pi} = (\pi_1, \dots, \pi_{\s N})$ 
in $\Path^{(i)}_{\, \bf i} \big(\mu, \lambda\big)$ with
weight $\wt(\pi_1) \cdots \wt(\pi_{\s N})
= a_1^{j_1} \cdots a_k^{j_k}$. Clearly

\[ \Path^{(i)}_{\, \bf i} \big(\mu, \lambda\big) \ = \
\bigsqcup_{\D \ \, {\bf j} \in \Bbb{Z}^k_{\s \geq 0}}  \ 
\Path^{(i)}_{\, \bf i} \big(\mu, \lambda; {\bf j}\big) \]

\bigskip
\noindent
If $\mu = \emptyset$ we shall use the short hand notation
$\Path^{(i)}_{\, \bf i}\big(\lambda\big)$ and 
$\Path^{(i)}_{\, \bf i}\big(\lambda; {\bf j}\big)$
instead.

\end{Def}

\begin{Rem}
In view of equation (2) and the Lindstr\"om lemma (3) it follows 
that 

\[ \Delta^{(i)}_{\mu, \lambda} \Big(
x_{i_1}(a_1) \cdots x_{i_k}(a_k) \Big) \ = \
\sum_{\D {\bf \pi} \in \Path^{(i)}_{\, \bf i} \big(\mu,\lambda\big) } 
\ \wt(\pi_0) \cdots \wt(\pi_{\s N})
\]

\noindent
where each ${\bf \pi} = \{\pi_0, \dots, \pi_{\s N}\}$ is a non-crossing
family of paths and where $\pi_n$ joins
source $u_n = \mu_n + i - n$ to sink $v_n = \lambda_n + i -n$
whenever $N \geq n \geq 0$ where $N = N_\lambda$.

\end{Rem}

\bigskip
\noindent
Consider now the case where $\mu = \emptyset$. As before
set $N = N_\lambda$ and let $u_n = i - n$ and $v_n =
\lambda_n + i -n$ whenever $N \geq n \geq 0$.
Choose ${\bf \pi} = \{\pi_1, \dots, \pi_{\s N}\}$ in
$\Path^{(i)}_{\, \bf i} \big(\lambda\big)$.
An index $t$ in $[1 \dots k]$ will be called an
{\it ascent} of a path $\pi$ in $\Gamma_{\bf i}({\bf a})$
if $\pi$ ascends along a diagonal edge within the
$\Gamma_{i_t}({\bf a})$ component.
Since $v_n - u_n = \lambda_n$ it must be the case that 
the path $\pi_n$ makes exactly $\lambda_n$ diagonal ascents as it
travels from $u_n$ up to $v_n$ in $\Gamma_{\bf i}({\bf a})$.
Let $t_{n,1} < \cdots < t_{n, \lambda_n}$ be the list
of the ascents of $\pi_n$. Since the paths are non-crossing
it also follows that the position of the $l$-th
ascent of $\pi_{n}$ must be directly under or to the right
of the $l$-th ascent of $\pi_{n-1}$; equivalently
$t_{n-1,l} \leq t_{n,l}$ whenever $\lambda_{n-1} \geq l
\geq 1$ and whenever $N \geq n \geq 1$.

\bigskip
\noindent
Record (in increasing order and from left to right)
the ascents of $\pi_n$ in the $n$-th row of $\lambda$
and repeat this for each $n$.
In this way we obtain
a semi-standard tableau $T_{\bf \pi}$ of shape 
$\lambda$ owing to the observations made above.
Moreover $\wt(\pi_1) \cdots \wt(\pi_{\s N})
= a_1^{j_1} \cdots a_k^{j_k}$ if and only if the
content of $T_{\pi}$ is ${\bf j}$. For example
the following pair ${\bf \pi} = \{\pi_0, \pi_1\}$
of (highlighted) non-crossing paths

\[ \xymatrix{
4 \ \circ \ar@{-->}[r]
& \circ \ar@{-->}[r]  
& \circ \ar@{-->}[r] 
& \circ \ar@{-->}[r]
& \circ \ar@{-->}[r] 
& \bullet \ 4 
\\
3 \ \circ \ar@{-->}[r] \ar@{-->}[ur] 
& \circ \ar@{-->}[r] 
& \circ \ar@{-->}[r] \ar@{-->}[ur]
& \circ \ar@{-->}[r]
& \bullet \ar@{-->}[r] \ar[ur]^{a_5}
& \circ \ 3 
\\
2 \ \circ \ar@{-->}[r]
& \bullet \ar[r] \ar@{-->}[ur] 
& \bullet \ar[r] 
& \bullet \ar@{-->}[r] \ar[ur]^{a_4}
& \circ \ar@{-->}[r] 
& \bullet \ 2 
\\
1 \ \bullet \ar@{-->}[r] \ar[ur]^{a_1}
& \circ \ar@{-->}[r] 
& \bullet \ar[r] \ar@{-->}[ur]
& \bullet \ar[r]
& \bullet \ar@{-->}[r] \ar[ur]^{a_5}
& \circ \ 1 
\\
0 \ \bullet \ar[r] 
& \bullet \ar@{-->}[r] \ar[ur]^{a_2}
& \circ \ar@{-->}[r]
& \circ \ar@{-->}[r] \ar@{-->}[ur]
& \circ \ar@{-->}[r]
& \circ \ 0 } \]

\bigskip
\noindent
would correspond to the tableau \ $T_{\bf \pi} = \young(145,25)$

\begin{Prop}
The map ${\bf \pi} \longmapsto T_{\bf \pi}$ is a bijection
between $\Path^{(i)}_{\, \bf i}\big(\lambda\big)$
and $\chess^{(i^*)}\big(\lambda \big)$
where $i^*$ is the parity of $i + i_1 + 1$. Moreover 
${\bf \pi} \in \Path^{(i)}_{\, \bf i}\big(\lambda; {\bf j}\big)$
if and only if $T_{\bf \pi} \in
\chess^{(i^*)}_{\s \bf j}\big(\lambda\big)$.

\end{Prop}

\begin{proof}
Since $\Gamma_{i_t}(a_t)$ allows ascents only from
vertices of parity $i_t$ and since $\Gamma_{\bf i}({\bf a})$
is constructed as an alternating concatenation
of such graphs it follows that 
(1) the parities of the ascents of any path must alternate 
when read from left to right
and (2) the parity of the first ascent equals the 
parity of $i + l + 1$ where $l$ is index of the source vertex. 
The source vertices 
of $\pi_0, \dots, \pi_{\s N}$ are $i, \dots, i-N_\lambda$
respectively. Consequently $T_{\bf \pi}
\in \chess^{(i^*)}\big(\lambda \big)$.

\bigskip
\noindent
Both injectivity and surjectivity follow from the fact
that a path in $\Gamma_{\bf i}({\bf a})$  
is uniquely determined by it source, sink, and
list of parity alternating ascents. The refinement
of these results to weight and content is obviously
true.

\end{proof}

\bigskip
\noindent
{\bf Proof of Theorem 2:}
Let $\lambda$ be a partition, let $i=0,1$ be a choice of 
parity, let ${\bf i} = (i_1, \dots, i_k)$ be an
alternating bit string in $\{0,1\}^k$, and let 
${\bf a} = (a_1, \dots, a_k)$ be a $k$-tuple in 
$\big(\Bbb{C}^*\big)^k$ then

\[ \begin{array}{ll}
{\mathlarger{\mathlarger{\mathlarger\Delta}}}^{(i)}_{ {\s \emptyset} , \lambda} 
\Big( x_{i_1}(a_1) \cdots x_{i_k}(a_k) \Big)
&= \ {\D \sum_{\D \ {\bf \pi} \in \Path^{(i)}_{\, \bf i} \big(\lambda\big) } 
\wt(\pi_0) \cdots \wt(\pi_{\s N}) } \\ \\
&= \ {\D \sum_{ \ \ \textstyle {\bf j} \in \Bbb{Z}_{\s \geq 0}^k } \ 
\sum_{\D {\bf \pi} \in \Path^{(i)}_{\, \bf i} \big(\lambda; {\bf j} \big) } 
\wt(\pi_0) \cdots \wt(\pi_{\s N}) } \\ \\
&= \ {\D \sum_{ \textstyle  \ \ {\bf j} \in \Bbb{Z}_{\s \geq 0}^k }  
\Big| \chess^{(i^*)}_{\s \bf j} \big( \lambda \big) \Big| \  
a_1^{j_1} \cdots a_k^{j_k} } \\ \\
&= \ {\D \sum_{ \textstyle  \ \ {\bf j} \in \Bbb{Z}_{\s \geq 0}^k }  
\chi \bigg( 
\mathcal{F}^{\s \Lambda}_{\bf  i^j } (M) \bigg) \ 
{a_1^{j_1} \cdots a_k^{j_k} \over {j_1! \cdots j_k!}} }
\end{array} \]

\begin{Rem}
It is worth mentioning here that if in the right hand side 
of the formula

\[ {\mathlarger{\mathlarger\Delta}}^{(i)}_{ {\s \emptyset} , \lambda} 
\Big( x_{i_1}(a_1) \cdots x_{i_k}(a_k) \Big) \ = \
\sum_{ \textstyle  \ \ {\bf j} \in \Bbb{Z}_{\s \geq 0}^k }  
\Big| \chess^{(i^*)}_{\s \bf j} \big( \lambda \big) \Big| \  
a_1^{j_1} \cdots a_k^{j_k} \]

\bigskip
\noindent
the chess condition is dropped the resulting sum 

\[ \sum_{ \textstyle  \ \ {\bf j} \in \Bbb{Z}_{\s \geq 0}^k }  
\Big| \Tab_{\s \bf j} \big( \lambda \big) \Big| \  
a_1^{j_1} \cdots a_k^{j_k} \]

\bigskip
\noindent
is precisely the definition of the {\it Schur polynomial}
$S_\lambda$ in the variables $a_1, \dots, a_k$; see
\cite{Fulton}. This observation taken together with
the remark made in section (4) about the $i$-Pieri rule
suggest that the block-Toeplitz minors $\Delta^{(i)}_{{\s \emptyset},
\lambda}$ should be viewed as generalized Schur polynomials.

\end{Rem}

\end{document}